\documentclass[12pt,reqno]{amsart}

\usepackage{palatino}
\usepackage[utf8]{inputenc}
\usepackage[LGR,T1]{fontenc}
\usepackage{comment}

\title[Reciprocal classes of random walks on graphs]{Reciprocal classes of random walks on graphs}
\date{, 2015}

\author{Giovanni Conforti}
\address{Institut für Mathematik der Universität Leipzig.Augustus Platz 10. 04109 Leipzig, Germany}
\email{giovanniconfort@gmail.com}

\author{Christian Léonard }
\address{Modal-X. Université Paris Ouest. Bât.\! G, 200 av. de la République. 92001 Nanterre, France}
\email{christian.leonard@u-paris10.fr}

\thanks{CL is partially supported by the French ANR projects GeMeCoD and STAB}

\usepackage{amssymb, amsmath, amsfonts, latexsym,amstext}
\usepackage{enumerate,graphicx}
\RequirePackage[colorlinks,linkcolor=blue,citecolor=blue,urlcolor=blue]{hyperref}
\usepackage[english]{babel}

\usepackage{color}
\usepackage{pst-fill,pst-grad,pst-plot,pst-eucl,pstricks-add,pst-node}

\newtheorem{theorem}{Theorem}
\newtheorem{lemma}[theorem]{Lemma}
\newtheorem{proposition}[theorem]{Proposition}
\newtheorem{corollary}[theorem]{Corollary}

\newtheorem{definition}[theorem]{Definition}
\newtheorem{definitions}[theorem]{Definitions}

\newtheorem{assumption}[theorem]{Assumption}

\theoremstyle{remark}
\newtheorem{remark}[theorem]{Remark}
\newtheorem{remarks}[theorem]{Remarks}

\numberwithin{theorem}{section}

\newcommand{\RR}{\mathbb{R}}
\newcommand{\Rn}{\mathbb{R}^n}

\newcommand{\1}{\mathbf{1}}

\newcommand\pf{_{\#}}

\newcommand{\as}{ \textrm{-}\mathrm{a.s.}}

\newcommand*{\cchi}{\raisebox{0.35ex}{\( \chi \)}}

\DeclareMathOperator{\supp}{supp}

\DeclareMathOperator{\proj}{proj}

\newcommand{\Boulette}[1]{\par\medskip\noindent $\bullet$\ Proof of #1.}

\renewcommand{\AA}{ \mathcal{A}}
\newcommand{\Aa}{\AA_\rightarrow}
\newcommand{\As}{\AA_\leftrightarrow}
\newcommand{\AaR}{\AA_\rightarrow^R}
\newcommand{\AsR}{\AA_\leftrightarrow^R}
\newcommand{\Ao}{\AA_0}
\newcommand{\Ap}{\Aa}

\newcommand\XX{\mathcal{X}}

\newcommand\XXX{\XX^2}
\newcommand\PX{\mathrm{P}(\XX)}

\newcommand\PXX{\mathrm{P}(\XXX)}

\newcommand\PO{\mathrm{P}(\Omega)}

\newcommand\OO{\Omega}

\newcommand\ii{{[0,1]}}

\newcommand\IXX{\int_{\XXX}}

\newcommand{\Rec}{\mathcal{R}}
\newcommand{\cc}{\mathbf{c} }
\newcommand{\ff}{\mathbf{f} }

\newcommand{\ww}{\mathbf{w}}
\newcommand{\ee}{\mathbf{e}}
\newcommand{\bes}{\begin{equation*}}
\newcommand{\ees}{\end{equation*}}
\newcommand{\beq}{\begin{equation}}
\newcommand{\eeq}{\end{equation}}

\renewcommand{\to}{\rightarrow}

\newcommand{\sz}{\sum _{z':z\to z'}}
\newcommand{\sw}{\sz}

\newcommand{\sx}{\sum _{x\in\XX}}

\newcommand{\jj}{{\bar j}}
\newcommand{\kk}{{\bar k}}

\newcommand{\zz}{z\to z' }
\newcommand{\zw}{\zz}

\newcommand{\YY}{\mathcal{Y}}
\newcommand{\iii}{[0,1)}

\begin{document}
\maketitle 
\tableofcontents

\begin{abstract}
The reciprocal class of a Markov path measures is the set of all the mixtures of its bridges. We give  characterizations of the reciprocal class of a continuous-time Markov random walks on a graph. Our main result is in terms of some reciprocal characteristics  whose expression only depends on the intensity of jump. We also characterize the reciprocal class by means of Taylor expansions in small time of some conditional probabilities. 

Our   measure-theoretical approach allows to extend significantly already  known results  on the subject. 
The abstract results are illustrated by  several examples.       

\end{abstract}


\section*{Introduction}

This article answers the questions: \emph{``When does a continuous-time random walk on a graph share its bridges with  a given Markov   walk?''} and \emph{`` What does a random walk share with its bridges?"}, both in terms of their intensities of jumps and of Taylor expansions in small time of probabilities of conditioned events. The precise answers are stated at  Theorem \ref{res-08} and Corollary \ref{res-12} which are our main results. 

The set of all path measures sharing the bridges of a given  Markov measure is called its reciprocal class.
In contrast with the existing literature about reciprocal classes, i.e.\ shared bridges, which    relies on  transition probabilities, in this paper  we adopt  a measure-theoretical approach: our main objects of interest are path measures, i.e. probability measures on the path space, rather than transition probability kernels.  It turns out that this is an efficient way for  solving our problem and  allows to extend significantly already  known results  on the subject.       

\subsection*{Notation}
Some notation is needed before bringing  detail about the above questions and their answers.
For any measurable space $Y,$ $ \mathrm{P}(Y)$ denotes the set of all probability measures on $Y.$ We denote the support of a probability measure $p$ by $\supp p.$
On a discrete space $A,$  we have $\supp p=\{a\in A: p(a)>0\}$ and for any probability measures  $p$ and $q$, $p$ is absolutely continuous with respect to $q:$   $p\ll q$, if and only if $\supp p\subset \supp q.$ 
\\
The support of a function $u\in\RR^A$ is defined as usual by $\supp u:= \left\{a\in A: u(a)\not =0\right\} .$ Functions with 	a finite support will be useful to define Markov generators without extra assumptions on the intensity of jumps.

We consider random walks from the unit time interval $[0,1]$ to a countable directed graph $(\XX,\AA)$ where only jumps along the set $\AA\subset \XXX$ of the  arcs of the graph  are allowed. The set of all the sample paths  is denoted by $\OO\subset \XX ^{ [0,1]}$. As we adopt a measure theoretical viewpoint, it is worth identifying the random processes and their laws on the path space. Consequently, any path measure  $P\in\PO$ is called a random walk.
\\
The canonical process on $\OO$ is denoted as usual by  $(X_t;0\le t\le 1)$.
 For any random walk $P\in\PO$  we denote 
 \begin{itemize}
 \item
 $P_0(dx):=P(X_0\in dx)\in\PX,$ its initial marginal;
 \item
  $P _{ 01}(dxdy):=P(X_0\in dx, X_1\in dy)\in\PXX$, its endpoint marginal;
  \item
   $P^x:=P(\cdot\mid X_0=x)\in\PO,$  the random walk conditioned to start at $x\in\supp P_0$;
   \item
   $P ^{ xy}:=P(\cdot\mid X_0=x,X_1=y)\in\PO,$ its $xy$-bridge with  $(x,y)\in\supp P _{ 01}$.
 \end{itemize}

\subsection*{Aim of the article}

We take  some Markov random walk $R\in\PO$ with an intensity of jumps $j:[0,1]\times \AA\to[0, \infty)$   and we assume that its initial marginal $R_0\in\PX$ has a full support. This random walk is our reference path measure. 

For comparison with the results of this article,  let us consider for a little while the set
  $$
 \mathcal{M}(j):= \left\{\sum _{ x\in\XX} \mu_0(x)\, R^x; \mu_0\in\PX\right\} \subset\PO,
 $$
 assuming  that for any $x$, $R^x$  is uniquely well defined. Obviously, for any $P\in\PO,$ the three following statements are equivalent:
\begin{enumerate}
\item
$P\in \mathcal{ M}(j);$
\item
$P^x=R^x,$ for all $x\in\supp P_0;$
\item
For any $x\in\supp P_0,$ $P^x $ is Markov with intensity $j.$
\end{enumerate}

 Rather than the collection of all  random walks $R^x$ conditioned by their starting point $x\in\XX$, our interest is in the bridges  of $R$. We define 
 \begin{equation}\label{eq-35}
 \Rec(j):= \left\{\sum _{ x,y\in\XX} \pi(x,y)\,R ^{ xy}; \pi\in\PXX: \supp\pi\subset\supp R _{ 01}\right\} \subset \PO
 \end{equation}
to be the convex hull of  all these bridges.  This set  is called the \emph{reciprocal class} of the intensity $j$. Since the reference random walk $R$ is Markov, so are its bridges $R ^{ xy}.$ But, in general a mixture of such bridges fails to remain Markov. However, any element $P$ of $\Rec(j)$
still satisfies the reciprocal property which extends the Markov property in the following way.

\begin{definition}[Reciprocal walk]\label{def-10}
A random walk $P\in\PO$ is said to be a \emph{reciprocal} walk if for any $0\le u\le v\le 1,$ $P(X _{[u,v]}\in \cdot\mid X _{[0,u]}, X _{[v,1]})=P(X _{[u,v]}\in \cdot\mid X_u, X_v).$
\end{definition}
The reciprocal property of a path measure is defined in accordance with the usual notion related to processes. 
Basic material about reciprocal walks is collected  at Appendix \ref{sec-A}.

 We see clearly that   for any $P\in\PO$ such that $\supp P _{ 01}\subset \supp R _{ 01},$ the two following statements are equivalent:
\begin{enumerate}
\item[(1)']
$P\in \mathcal{R}(j);$
\item[(2)']
$P ^{ xy}=R ^{ xy},$ for all $(x,y)\in\supp P _{ 01}.$
\end{enumerate}
The aim of the article is to provide an analogue of statement (3) above.  Indeed,  Theorem \ref{res-08} states that (1)' and (2)' are equivalent to
\begin{enumerate}
\item[(3)']
For any $(x,y)\in\supp P _{ 01},$ $P^{xy} $ is Markov and\begin{equation}\label{eq-34}
\cchi[k^{xy}]=\cchi[j],
\end{equation}
\end{enumerate}
where $k^{xy}$ is the intensity of $P^{xy}$ and $\cchi[k^{xy}], \cchi[j]$ are described at Definition \ref{def-01}.
For this reason, $\cchi[j]$ is called the \emph{characteristic} of the reciprocal class $\Rec(j).$  

There are other random walks $P$  than $R$  in
$\Rec(j)$ that are Markov. In this case, for every $x\in\supp P_0,$ the intensity $k^x$ of   $P^x$   does not depend on $x$. When $k^x$ depends explicitly on $x$, $P$ is not Markov; this is the case for most of  the elements of $\Rec(j)$.

\subsection*{Variational processes}

Beside the interest in its own right of the description of the convex hull of the bridges of some Markov dynamics, there exists a stronger motivation for investigating the reciprocal class of a Markov process.  Suppose that you observe two large samples of non-interacting particles systems 1 and 2 with two distinct endpoint distributions (i.e.\ empirical measures of the couples of initial and final positions). \emph{Are these two random systems driven by the same force field?}  In mathematical terms, you want to know if the path measures $P_1$ and $P_2$ corresponding to the systems 1 and 2 belong to the same reciprocal class. To understand this equivalent statement, let us provide some comments.
This question is rooted into a problem addressed by Schrödinger in the early 30's in the articles \cite{Sch31,Sch32}. 

\subsubsection*{Schrödinger problem}

Consider a large number $N$ of independent particles labeled by $1\le i\le N$ and moving according to some Markov dynamics described by  the reference path measure $R\in\PO.$ In Schrödinger's papers, $R$  is the law of a Brownian motion on $\XX=\Rn$ and $\OO=C([0,1],\Rn).$  Suppose that at time $t=0$, the particles are distributed according to a profile close to some distribution $ \mu_0\in\PX.$ In modern terms, this means that the empirical measure $ \frac{1}{N}\sum _{ 1\le i\le N} \delta _{ X_i(0)}$ of the initial sample is weakly close to $ \mu_0,$  where $ \delta_a$ is the Dirac measure at $a$ and $t\mapsto X_i(t)$ describes the random motion of the $i$-th particle. Suppose that at time $t=1$ you observe that the whole system is such that its distribution  profile $ \frac{1}{N} \sum _{ 1\le i\le N} \delta _{ X_i(1)}$ is weakly close to some $ \mu_1\in\PX$  far away from the expected profile $ \mu_0 e^L\in\PX$  which is predicted by the law of large numbers. Here, $L$  is the Markov generator of $R$. Schrödinger asks what is the most likely trajectory of the whole particle system conditionally on this very rare event. As translated in modern terms by Föllmer in \cite{Foe85} using large deviations technics, the empirical measure $ \frac{1}{N}\sum _{ 1\le i\le N} \delta _{ X_i}$ weakly tends in $\PO$ to the unique solution of the entropy minimization problem
\begin{equation}\label{eq-27}
H(P|R)\to \mathrm{min};\qquad P\in\PO: P_0= \mu_0, P_1= \mu_1,
\end{equation}
where $H(P|R):= E_P\log(dP/dR)\in [0, \infty]$ is the relative entropy of $P$ with respect to the reference measure $R$. A recent review of the Schrödinger problem is proposed in  \cite{Leo12e}.

\subsubsection*{A stochastic analogue of Hamilton's principle}

In the  case where $R$ is a  Brownian motion, any $P\in\PO$ with $H(P|R)$  finite is the solution of a martingale problem associated with some  adapted drift field $ \beta^P$ such that $E_P\int _{ [0,1]}| \beta^P_t|^2\,dt< \infty$ and
$	
H(P|R)= E_P\int _{ [0,1]}| \beta^P_t|^2/2\,dt.
$	
As $ |\beta^P_t|^2/2$ is a kinetic energy,   $H(P|R)$ is an average kinetic  action and the minimization  problem  \eqref{eq-27} appears to be a stochastic generalization of the usual Hamilton variational principle, see for instance \cite{Leo12e} and the references therein.

\subsubsection*{A natural extension of Schrödinger's problem}

Problem \eqref{eq-27} also admits the following natural extension
\begin{equation}\label{eq-33}
H(P|R)\to \mathrm{min};\qquad P\in\PO: P _{ 01}= \pi
\end{equation}
with $ \pi\in\PXX$ a prescribed endpoint distribution. In some sense, the entropy minimization  problem \eqref{eq-33} is the widest stochastic extension of  the classical Hamilton least action principle.

\subsubsection*{The connection with the reciprocal class $\Rec[j]$}

To see the connection with the bridges of $R,$ let us go back to   our discrete set of vertices $\XX$  and  note that both  problem  \eqref{eq-27} with the prescribed marginals $ \mu_0= \delta_x$ and $ \mu_1= \delta_y$ and   problem \eqref{eq-33} with $ \pi= \delta _{ (x,y)},$ admit the unique solution $$P=R ^{ xy}.$$ This is a simple consequence of the additive  decomposition formula
\begin{equation*}
H(P|R)=H(P _{ 01}|R _{ 01})+\sum _{ x,y\in\XX} P _{ 01}(x,y) H(P ^{ xy}|R ^{ xy})
\end{equation*}
applied with $P _{ 01}= \delta _{ (x,y)}.$ One can interpret the bridge $R ^{ xy}$ as the stochastic analogue of a minimizing geodesic between $x$ and $y$.  
We also see with the additive  decomposition formula that the unique solution  of \eqref{eq-33} with a general endpoint distribution $ \pi\in\PXX$ such that $H( \pi|R _{ 01})< \infty$
 is 
\begin{equation*}
P=\sum _{ x,y\in\XX} \pi(x,y)\, R ^{ xy}.
\end{equation*}
Therefore, the reciprocal class $\Rec(j)$ appears to be essentially  the set of all the solutions of the stochastic variational problem \eqref{eq-33} when $ \pi\in\PXX$ describes all the possible endpoint distributions. In addition, we see with \eqref{eq-34} that the bridges of $R$ all satisfy
\begin{equation*}
\cchi[j^{ xy}]=\cchi[j],\quad \forall (x,y)\in \supp R _{ 01}.
\end{equation*}
\emph{This indicates that the reciprocal characteristic $ \cchi[j]$ encrypts  the underlying stochastic Lagrangian associated with the stochastic action minimization problem \eqref{eq-33}.}

This point of view is developed in a diffusion setting by Zambrini and the second author in \cite{LZ14}.  It will be explored in the present setting of random walks on graphs in a forthcoming paper.

At the present time very little is known about the solutions of the variational problems \eqref{eq-27} and \eqref{eq-33} in the setting of random walks on graphs. This paper is a contribution in this direction.

\subsection*{Literature}

One year after    Schrödinger's article   \cite{Sch31}, Bernstein introduced in \cite{Bern32} the reciprocal property as a notion that extends Markov property and is respectful of the  time reversal symmetry. It was further developed four decades later by Jamison  \cite{Jam74,Jam75}.
Relying on Jamison's approach to reciprocal processes, Clark  proved in \cite{Cl91} a conjecture of Krener \cite{Kre88} who proposed a characterization of the reciprocal class of a Brownian diffusion process in terms of an identity of the type $\cchi[ \beta^P]=\cchi[0]$ where $ \beta^P$ is the drift of the Markov diffusion process $P$ (recall the discussion below Eq.\,\eqref{eq-27} for the notation $ \beta^P$). Clark called $\cchi[ \beta^P]$ a   reciprocal \emph{invariant}. In the present article, we prefer naming ``reciprocal \emph{characteristics}'' the analogous quantities $\cchi[k]$.\footnote{In the contexts of classical, quantum and stochastic mechanics (the latter being taken in its wide acception including Nelson's mechanics \cite{Nel85}, Euclidean quantum mechanics \cite{CZ08} or hydrodynamics  where the evolution is deterministic but the initial state is described by a probability measure), the term ``invariant'' refers to  conserved quantities as time varies.} In view  of the previous discussion about variational processes, it is not surprising that the reciprocal characteristics play a distinguished role when looking at reciprocal processes as solutions of second order stochastic differential equations. This is investigated by Krener, Levy and Thieullen in  \cite{KL93,Th93,Kre97}.

Characterization of reciprocal classes can also be stated in terms of stochastic integration by parts formulas, often called duality formulas. This was investigated by  Roelly and Thieullen in \cite{RT02,RT05} for diffusion processes. In the specific context of counting random walks, this is  done in \cite{CLMR14}, a paper by Murr, Roelly and the authors of the present article. This was extended by Dai Pra, Rœlly and the first author in \cite{CDPR14} for compound Poisson processes and in \cite{CR14} for random walks on Abelian groups. A main idea of \cite{CDPR14,CR14}  is to exploit the translation-invariant structure of the underlying graph to characterize the reciprocal classes through integration by parts formulas where the derivation  measures the variation when adding a random closed walk to the canonical process, see \cite[Thm.\,3.3]{CDPR14}  and  \cite[Thm.\,13]{CR14}. In all these cases, reciprocal characteristics play a major role.
 If the graph is not assumed to be invariant with respect to some group transformations, such as translation-invariance and time homogeneity, there is no way of thinking of a natural derivative. 
Since we do not assume any invariance in the present article, the integration by parts approach is not investigated.

It is worthwhile to note that, except for \cite{RT02} and  the recent papers  \cite{CLMR14,CDPR14,CR14,LZ14}, in the whole literature on the subject, only the \emph{Markov members} of the reciprocal class, that is the solutions of the original Schrödinger problem \eqref{eq-27} as the marginal constraints vary, are characterized. In the present article, following a strategy close to  \cite{LZ14}'s one,  we give a characterization of the \emph{whole} reciprocal class, that is  the set of solutions of \eqref{eq-33} as $\pi$ varies, under very few restrictions  on the reference random walk.

\subsection*{Outline of the paper}
Next Section \ref{sec-prem} is devoted to some preliminaries about directed graphs, Markov walks and their intensities. We also state our main hypotheses which are Assumption \ref{as-03} and Assumption \ref{as-01} and define carefully the reference path measure $R$. Our main results are stated at Section \ref{sec-main results}. They are Theorems \ref{res-08} and \ref{res-06}, together with their Corollary \ref{res-12}. Theorem \ref{res-08} is a rigorous version of statement  (3)'  while Theorem \ref{res-06} provides an interpretation of the reciprocal characteristics in terms of Taylor expansions in small time  of  some conditional probabilities. Their proofs are done at Section \ref{sec-proofs}. The key preliminary result is Lemma \ref{res-03} whose proof is based on the identification of two expressions for the Radon-Nikodym derivative $dR^{xy}/dR$. Several examples are treated at Section \ref{sec-examples}. We have  a look at:  birth and death processes, some planar graphs, the hypercube, the complete graph and some Cayley graphs. We calculate their reciprocal characteristics and sometimes solve the associated characteristic equation. Finally, there are appendix sections devoted to reciprocal random walks and to closed walks on a directed graph.

\section{Preliminaries}\label{sec-prem}

This section is devoted to some preliminaries about directed graphs, Markov walks and their intensities. Our main hypotheses are stated below at Assumption \ref{as-03} and Assumption \ref{as-01}. Much of the material in this section is required for the definition of the reciprocal characteristics and the  statements of our results at Section \ref{sec-main results}.
\subsection*{Directed graphs}
Let $\XX$ be a countable set and $ \AA\subset\XXX$.
 The \emph{directed graph}  associated with $\AA$ is defined by means of the relation $\to$, meaning that for all $z,z'\in\XX$ we have $\zz$ if and only if $(z,z')\in \AA.$   We denote $(\XX,\to)$ this directed graph, say that any $(z,z')\in \AA$ is an arc and  write $(\zz)\in \AA$ instead of $(z,z')\in \AA.$ 

 We are concerned with directed graphs $(\XX,\to )$ satisfying the following:
\begin{assumption}\label{as-03}
The directed graph $(\XX,\to)$ satisfies the following requirements.
 \begin{enumerate}
 \item $\AA$ is symmetric:  $(x \to y )\in \AA \Rightarrow (y \to x ) \in \AA$.
\item It is connected: for any $x,y$ there exists a directed walk from $x$ to $y.$ 
\item
It is of bounded degree.
\item
It has no loops, meaning that for all $z\in\XX,$ $(z\to z)\not\in \AA$.
\end{enumerate}
\end{assumption}

Let us also give some definitions which are necessary to state our main results. Here, we use some standard vocabulary in graph theory. For more detail about walks, simple walks, closed walks, we refer to Appendix \ref{sec-B}.

\begin{definitions}[Tree and basis of closed walks]\label{defs-01}
\ \begin{enumerate}[(a)]
\item We call \emph{tree} a symmetric connected subgraph $\mathcal{T}$ of $(\XX,\to)$ with no closed walks of length at least three. In this paper a tree is always a \emph{spanning} tree, in the sense that it connects any pair of vertices in  $ \XX$.
\item Let $\mathcal{T}$ be a tree. If $(x \to y) \notin \mathcal{T}$, we denote $\mathbf{f}_{x \to y}$ 
the closed walk obtained by concatenating $x \to y$ with the only simple directed walk from $y$ to $x$ in $\mathcal{T}$. 
\item Let $\mathcal{T}$ be a tree. A $\mathcal{T}$-\emph{basis of  the closed walks} of $(\XX, \to)$ is any subset $\mathcal{C}$ of closed walks of the form: 
\bes
\mathcal{C} = \mathcal{C}_0 \cup \mathcal{E}
\ees
 where 
 \bes
 \mathcal{E}:= \{ (x \to y \to x), \ (x\to y) \in \XXX \}
 \ees
 stands for the set of all edges
 and $\mathcal{C}_{0}$ is obtained by choosing for any $(x \to y \to x) \in \mathcal{E} \setminus \mathcal{T}$ exactly one among $\mathbf{f}_{x \to y}$ and $\mathbf{f}_{y \to x}$
\end{enumerate}
\end{definitions}
With the above construction,  all elements of a basis of closed walks are indeed closed walks. 
We refer to \cite[Sec. 2.6]{BM07} for the notion of cycle basis of  an undirected graph.

\subsection*{Random walk on a graph}


The countable set $\XX$ is equipped with its discrete topology.
We  look at  continuous-time random paths on $(\XX,\to)$ with finitely many jumps on the bounded time interval $[0,1]$. The corresponding  \emph{path space} $\OO\subset \XX ^{\ii}$  consists of all càdlàg   piecewise constant paths $\omega=(\omega_t)_{0\le t\le1}$ on $\XX$ with  \emph{finitely many jumps} such that $ \omega _{ 1^-}= \omega_1$ and for all $ t\in (0,1),$ $\omega _{t^-}\ne \omega _{t} $ implies that $\omega _{t^-}\to\omega_t$.
 It is equipped with the canonical $\sigma$-field  generated by the canonical process. 
 
 \begin{definition}
  We call any probability measure  on $\OO$  a  \emph{random walk} on $(\XX,\to)$.
 \end{definition}
  This is not  the customary usage, but it turns out to be  convenient. As a probability measure, it specifies the behavior of a piecewise constant continuous-time random process  that  may not be Markov.

\subsubsection*{Notation related to random walks}
As usual, the canonical process $X=(X_t)_{t\in\ii}$ is defined for each $t\in\ii$ and $\omega=( \omega_s) _{ 0\le s\le 1}\in\OO$ by $X_t(\omega)=\omega_t\in\XX.$
For any $\mathcal{I}\subset\ii$ and any random walk $P\in\PO,$ we denote $X_\mathcal{I}=(X_t)_{t\in \mathcal{I}}$ and the push-forward measure $P_\mathcal{I}=(X_\mathcal{I})\pf P.$ In particular, for any $0\le t\le1,$ $P_t=(X_t)\pf P\in\PX$ denotes the law of the position $X_t$ at time $t$ and $P _{ 01}:=P _{ \left\{0,1\right\} }$ denotes the law of the endpoint position $(X_0,X_1).$ Also, for all
 $0\le r\le s\le1,$ $X_{[r,s]}=(X_t)_{r\le t\le s}$ and  $P_{[r,s]}=(X_{[r,s]})\pf P.$ 
 \\
 We are mainly interested in  bridges.
 In a general setting, one must be careful because  the bridge 
 $
 P ^{ xy}
 $
is only defined $P _{ 01}$-almost everywhere. But in the present case where the state space $\XX$ is countable, the kernel $(x,y)\mapsto P ^{ xy}$ is defined \emph{everywhere} on $\supp P _{ 0 1}$, no almost-everywhere-precaution is needed when talking about bridges. In particular, any random walk $P$ disintegrates as
\begin{equation*}
P= \sum _{ (x,y) \in \supp P _{ 01}}P_{01}(x,y)P ^{ xy}\in\PO.
\end{equation*}
 In the whole paper, the letters $x$ and $y$ are devoted respectively to the initial and final states of random walks. Current states are usually denoted by $z,z'\in\XX.$

\subsection*{Markov walk}

The Markov  property of a path measure is defined in accordance with the usual notion related to processes. 

\begin{definition}[Markov walk]\
A random walk $P\in\PO$ is said to be a \emph{Markov} walk if for any $0\le t\le1,$ 
$P (X _{[t,1]}\in \cdot\mid X _{[0,t]})=P (X _{[t,1]}\in \cdot\mid X _t).$
\end{definition}

A Markov walk $P\in\PO$ is the law of a continuous-time Markov chain with its sample paths in $\OO.$ In our setting, all the Markov walks to be encountered will be associated with some intensity of jumps\footnote{This is maybe the case for any Markov walk with finitely many jumps, but we shall not need to investigate such a general existence result.}  $k:[0,1)\times \AA\to[0, \infty)$  which gives rise  to the infinitesimal generator  that acts on any real function $u\in\RR^\XX$ with a finite support via the formula
\begin{equation*}
K_t u(z)=\sw k(t,\zw)\,(u_{z'}-u_z), 
\quad  z\in\XX, t\in[0,1).
\end{equation*}
The random walk $P\in\PO$ is such that for any real function $u\in\RR^\XX$ with a finite support, the process $u(X_t)-\int_0^t K_su(X_s)\,ds$ is a local $P$-martingale with respect to the canonical filtration.
The average frequency of jump from $z$ at time $t$ is
\begin{equation}\label{eq-06}
\kk(t,z):=\sw k(t,\zw ).
\end{equation}
Note that this is a finite number because it is assumed that $(\XX,\to)$ is locally finite.
The sample paths of the random walk  $P$ admit a version in $\OO$ if and only if $k(t,\zw )$ is $t$-measurable for all $z$ and $z'$ and 
\begin{equation*}
\int _{[0,1)} \kk(t,X_t)\,dt<\infty,\ P\as
\end{equation*}
Indeed,  this estimate means that the canonical process performs  $P\as$ finitely many jumps.  

In the situation where $k$ doesn't depend on $t$, the dynamics of the random walk is described as follows. Once  at site $z,$ the walker waits during a random time with exponential law with parameter  $\kk(z)$ and then jumps onto $z'$ according to the probability measure $\kk(z) ^{-1}\sw k(\zw ) \delta _{ z'}$ where $\delta _{ z'}$ stands for the Dirac measure at $z'$, and so on; all these random events being mutually independent.

Note that $k(t,\cdot)$ and $K_t$
 are only defined for $t$ in the semi-open interval $[0,1).$ The reason for this is that we are going to work with mixtures of bridges and the forward intensity of a bridge is singular at $t=1.$

\subsection*{The reference intensity of jumps and the reference random walk}

We introduce a  random walk $R\in\PO$  with intensity of jumps $j$. Both $R$ and $j$ will serve as reference path measure and intensity. 

\begin{assumption}\label{as-01} \ 
The Markov  jump intensity $j:\ii\times\AA\to[0, \infty)$ verifies the following requirements.
\begin{enumerate}[(1)]
\item
The estimate
\begin{equation}\label{eq-11}
\sup _{ t\in\ii,z\in\XX}\jj(t,z)<\infty
\end{equation}
holds where, as in \eqref{eq-06},
$	
\jj(t,z):=\sw j(t,\zw )
$	
stands for the average frequency of jump from $z$ at time $t$.
\item
 $j$ is continuously $t$-differentiable, i.e.\ for any    $(\zz) \in \AA$ the function $t\mapsto j(t,\zw )$ is continuously differentiable on  $\ii$.
\item  $j$ is positive: 
\begin{equation*}
\forall \, t \in [0,1], \zz \in \mathcal{A}, \quad j(t,\zz)>0
\end{equation*}
\end{enumerate}
\end{assumption}


The Assumption \ref{as-01} (1)  implies that
for each $x\in\XX,$ there exists a unique solution $$R^x\in\PO$$ to the martingale problem with initial marginal $ \delta_x$ associated with the generator $L=(L_t)_{ 0\le t\le 1}$ defined for all finitely supported functions $u$ by
\begin{equation*}
L_t u(z)=\sw j(t,\zw )\,(u_{z'}-u_z), 
\qquad  t\in \ii, z\in\XX.
\end{equation*} 
Because of the fact that $(\XX,\to)$ is connected and Assumption \ref{as-01} the bridge $R ^{xy}:=R^x(\cdot\mid X_1=y)$ of $R^x$  is well defined for all $x,y\in\XXX $. 
The reference random walk  is defined by:
 \begin{equation*}
 R:=\sx R_0(x)\,R^x\in\PO
 \end{equation*}
 where the initial marginal $R_0$ is any probability measure on $\XX$ with a full support, i.e.\ $\supp R_0=\XX.$ 

 \section{Main results}\label{sec-main results}
 
Before stating the main results of the article, we still need to introduce two objects which are related to the notion of reciprocal walk, see Definition \ref{def-10} for this notion.

\subsection*{Reciprocal class}
The reciprocal class $\Rec(j)$ is defined at \eqref{eq-35}. It is the main object of our study.

\begin{proposition}\label{res-01}
The reciprocal class
$\Rec(j)$  is a set of reciprocal walks  in the sense of Definition \ref{def-10}, which are  absolutely continuous with respect to $R$.
\end{proposition}

\begin{proof}
 By its very definition, $\Rec(j)$ is the subset of all convex combinations of the bridges of the Markov walk $R.$ Remark \ref{rem-01}(d) tells us that any $P\in\Rec(j)$ is reciprocal. Moreover, Proposition  \ref{res-02} tells us that $P\ll R$ because  $\supp P _{ 01}=\supp \pi\subset \supp R _{01}= \XX^2.$
\end{proof}

\begin{remark}
Since any element of $\Rec(j)$ is absolutely continuous with respect to $R$, by Girsanov's theory it admits  a predictable intensity of jumps, see \cite[Thm.\,4.5]{Jac75}. This will be used constantly in the rest of the article.
\end{remark}

\subsection*{Reciprocal characteristics}

We are going to give a   characterization of the elements of $\Rec(j)$ in terms of  \textit{reciprocal characteristics} which we introduce right now.

\begin{definitions}[Reciprocal characteristics of a Markov random walk] \label{def-01}
Let $k$ be a jump intensity
which is assumed to be continuously $t$-differentiable and positive, i.e.\ for any    $(\zz) \in \mathcal{A}$ the function $t\mapsto k(t,\zw )$ is continuously differentiable on the semi-open time interval $[0,1)$ and positive on $(0,1)$.
\begin{enumerate}[(a)]
\item We define for  all $t \in (0,1)$ and all  $(\zw )\in \mathcal{A}$, 
\begin{equation*}
\cchi_{a}[k](t,\zw ):= 
\partial_t \log k(t,\zw ) + \kk(t,z')-\kk(t,z)
\end{equation*}
where $\kk$ is defined at \eqref{eq-06}.
\item
 We define  for all $t \in (0,1)$ and any closed walk $\cc=(x_0\to \cdots\to x_{\vert \cc \vert}=x_0)$ on  $(\XX,\to)$ 
\begin{equation*}
\cchi_c[k](t,\cc):= \prod_{i=0}^{|\cc|-1} k(t,x_i \to x_{i+1}) .
\end{equation*}
 See Definition \ref{def-02} for the notion of closed walk.
\item
 We call $\cchi[k]=(\cchi_a[k],\cchi_c[k])$ the reciprocal characteristic of $k$.
\\
The term  $\cchi_a[k]$ is  the \emph{arc component} and   $\cchi_c[k]$  is the \emph{closed walk component} of $\cchi[k]$. 
\end{enumerate}
\end{definitions}
Note that under our regularity assumption on $k$, $ \partial_t$ acts on a differentiable function:  $\cchi[k]$ is well defined.

\subsection*{The main results}

They are stated at Theorems \ref{res-08}, \ref{res-06} and Corollary \ref{res-12}.
Theorem \ref{res-08} gives a characterization of the reciprocal class of $j$ in terms of the reciprocal characteristics. Theorem \ref{res-06} provides an interpretation of the reciprocal characteristics of a reciprocal walk  by means of short-time asymptotic expansions of some conditional probabilities. Putting together these theorems leads us to Corollary \ref{res-12} which states a characterization of the reciprocal class in terms of these short-time asymptotic expansions.

\begin{theorem}[Characterization of $\Rec(j)$]\label{res-08}
We suppose that $(\XX,\to)$ satisfies Assumption \ref{as-03} and $j$ satisfies Assumption \ref{as-01}. 
\\
A random walk 
$P\in \PO$   belongs to $\Rec(j)$ if and only if 
 the following assertions hold for all $(x,y) \in \supp P_{01}$.
\begin{enumerate}[(i)]
\item 
The bridge $P^{xy}$ is Markov\footnote{But this doesn't imply that $P$ is Markov.}, $P^{xy}\ll R^{xy}$  and its intensity $k^{xy}$ is  $t$-differentiable and positive on $(0,1)$.

\item 
There exists a tree $\mathcal{T}$  such that for any $t\in [0,1)$ and any $(\zz)\in \mathcal{T}$, we have:
\begin{equation}\label{eq-42}
\cchi_{a}[k^{xy}](t,\zw ) = \cchi_a[j](t, \zw ).
\end{equation}
\item 
There exists a $\mathcal{T}$-basis of closed walks $\mathcal{C}$  such that for any $t\in (0,1)$ and any  $\cc \in \mathcal{C}$, we have
\begin{equation}\label{eq-10}
\cchi_c[k^{xy}] (t,\cc)= \cchi_c [j](t,\cc).
\end{equation}
\end{enumerate}
\end{theorem}

If $P$ is itself Markov, Theorem \ref{res-08}  simplifies. Indeed, we do not have to check conditions (ii) and (iii) for any $(x,y) \in \supp(P_{01})$, but we can simply test their validity on the jump intensity of $P$. We shall construct in Section \ref{sec-examples} different time-homogeneous Markov walks in the same reciprocal class.
\begin{corollary}[Markov elements of a reciprocal class]\label{res-14}
We suppose that $(\XX,\to)$ satisfies Assumption \ref{as-03} and $j$ satisfies Assumption \ref{as-01}. Let $P$ be a Markov walk of intensity $k$, which is continuously differentiable and positive on $(0,1)$. Then $P \in \Rec(j)$ if and only if items (ii) and (iii) of Theorem \ref{res-08} hold with $k$ instead of $k^{xy}$.
\end{corollary}
\begin{proof}
Since $P^{xy} \in \Rec(k)$ then for all $x,y$ an application of Theorem \ref{res-08} tells that 
$\cchi_a[k^{xy}] \equiv \cchi_{a}[k] $ over $\mathcal{T}$ and $\cchi_c[k^{xy}] \equiv \cchi_{c}[k] $ over $\mathcal{C}$ . But then checking (ii) and (iii) for $k^{xy}$ is the same as checking it for $k$.
\end{proof}
The intensity of a bridge can be found by solving the following equation, which we call \emph{characteristic equation}.

\begin{corollary}[Characteristic equation]\label{res-15}
Let $\mathcal{C}$ be a basis for the closed walks and $\mathcal{T}$ be a rooted tree. The intensity $j^{xy}$ of the $xy$ bridge is the only classical solution of:
\begin{equation}\label{eq-09}
\begin{cases}
\cchi_{c}[j^{xy}](t,\cc) = \cchi_{c}[j](t,\cc) , \quad & t \in (0,1) , \ \cc \in \mathcal{C} \\
\cchi_{a}[j^{xy}](t,\zz) = \cchi_{a}[j](t,\zz), \quad & t \in (0,1),\ (\zz) \in \mathcal{T}
\end{cases}
\end{equation}
subject to the boundary conditions
\begin{equation}\label{e-400}
\int_{0}^{1} j^{xy}(t,z)dt = \begin{cases} - \log R^x(y)+\int_{0}^{1} \jj(t,y)dt \quad & \mbox{if $z=y$} \\ - \infty , \quad & \mbox{otherwise}
\end{cases} 
\end{equation}
\end{corollary}
\begin{proof}
Equation \eqref{eq-09} is a direct consequence of (i),(ii),(iii) of Theorem \ref{res-08}. Concerning the boundary conditions, we will use Lemma \ref{res-03}. From the HJB equation \eqref{eq-19} we deduce that 
\bes
\partial_t \phi_t(z) = \jj(t,z) - \jj^{xy}(t,z)
\ees
Integrating in time the last equality and using the boundary conditions for $\phi$, see \eqref{eq-45}, it follows that a valid set of boundary conditions is given by \eqref{e-400}
\end{proof}
The reciprocal  characteristics come with a natural probabilistic interpretation which is expressed in terms of short-time asymptotic for the distribution of  bridges. 
We shall show that they can be recovered as quantities related to  Taylor expansions as $h>0$ tends to zero of conditional probabilities of the form 
 $
P(X _{ [t,t+h]}\in\cdot\mid X_t, X _{ t+h}).
 $
This is the content of Theorem \ref{res-06} below.
Let us introduce the notation  needed for its statement. For any integer $k\ge1$ and any $0\le t<1,$ we denote by $T^t_k$ the $k$-th instant of jump after time $t$. It is defined for $k=1$ by $T^t_1:=\inf \left\{s\in(t,1]:X _{ s^-}\not =X_s\right\} $ and for any $k\ge2$ by $T^t_k:=\inf \left\{s\in (T^t _{ k-1},1]: X _{ s^-}\not =X_s\right\} $ with the convention $\inf\emptyset=+ \infty.$

\begin{theorem}[Interpretation of the characteristics]\label{res-06}
We suppose that $(\XX,\to)$ satisfies Assumption \ref{as-03} and $j$ satisfies Assumption \ref{as-01}. 
Let  $P$ be any  random walk  in $\Rec(j)$.
\begin{enumerate}[(a)]
\item
For any $t \in (0,1)$, any $(\zz)\in \mathcal{A}$  
\begin{equation}\label{eq-02}
\begin{split}
P(T^t_1\leq t+h/2 \mid X_t=z,&X _{ t+h}=z', T^t_2>t+ h)\\
	&=\frac{1}{2} - \frac{h}{8} \cchi_a[j](t,\zw ) +o _{ h\to0^+}(h).
\end{split}
\end{equation}
\item

For any $t \in (0,1)$ and any closed walk $\cc$, we have
\begin{equation}\label{eq-05}
\begin{split}
P\Big((X_t\to X _{ T^t_1}\to\cdots\to X _{ T^t _{ |\cc|}}=X_t)=\cc,T^t _{ |\cc|}&<t+h< T^t _{ |\cc|+1}\mid  X_t=X _{ t+h}\Big)
	\\&=\cchi_c[j](t,\cc) h ^{ |\cc|}/|\cc|!+o_{ h\to0^+}(h ^{ |\cc|}).
\end{split}
\end{equation}
\end{enumerate}
\end{theorem}

\begin{remark}
The link between reciprocal characteristics and short time asymptotic for continuous-time random walks was sketched in \cite{CDPR14} in the particular case when the graph is a lattice and the intensity is space-time homogeneous. Concerning the diffusion case, it is due to Krener \cite{Kre97}.
\end{remark}

\begin{remark}[Reciprocal characteristics and the concentration of measure phenomenon]
The interpretation of the characteristics given in Theorem \ref{res-06} can be used to make quantitative statements on the behavior of a bridge whose lifetime is very short. However, it is a very natural question to ask what happens for non-asymptotic time scales, and if one is able to give natural conditions on the characteristics under which the fluctuations of a bridge can be controlled. This question is the object of the forthcoming work \cite{C15}. In particular, the connection between reciprocal characteristics and the concentration of measure phenomenon is made there.
\end{remark}

In the same spirit that a Markov walk is specified by the Markov property and its jump intensity which can be obtained as the limit in small time of a  conditional expectation,  we obtain the following characterization of $\Rec(j)$.  

\begin{corollary}[Short-time expansions characterize $\Rec(j)$]\label{res-12}
A random walk 
$P\in \PO$   belongs to $\Rec(j)$ if and only if 
 the following assertions hold .
\begin{enumerate}[(i)]
\item 
$P$ is reciprocal in the sense of Definition \ref{def-10}, and for all $x,y \in \supp P_{01}$, $P^{xy}\ll R^{xy}$  and its intensity $k^{xy}$ is  $t$-differentiable on $[0,1)$.
\item
There exists a tree $\mathcal{T}$ such that for any $t \in (0,1)$, any $(\zz)\in \mathcal{T}$, the identity  \eqref{eq-02} is satisfied.
\item
There exists a $\mathcal{T}$-basis of closed walks $\mathcal{C}$  such that for any $t\in [0,1)$ and any  $\cc \in \mathcal{C}$, we have
that  the identity \eqref{eq-05} is satisfied.
\end{enumerate} 
\end{corollary}

\begin{proof}
The necessary condition  is a direct consequence of  Theorems \ref{res-08} and \ref{res-06}. For the sufficient condition, all we have to show is that the properties (a) and (b) of Theorem \ref{res-06} respectively imply the properties (ii) and (iii) of Theorem \ref{res-08}.
\\
First we observe that, thanks to the reciprocal property, (a) and (b) extend to any bridge of $P$. 
The same calculations as in  Theorem \ref{res-06}'s proof at page \pageref{sec-pf-res-06} show that  replacing $R$ by $P^{xy}$ and $j$ by $k ^{xy}$ lead to the same conclusions with $k^{xy}$ instead of $j$. 
It remains to compare the resulting expansions to conclude that \eqref{eq-42} and \eqref{eq-10} are satisfied.
\end{proof}

\begin{remark}[Second order calculus for diffusion processes]
In the diffusion case, reciprocal processes have been used to develop a "second order calculus" for diffusions. In particular Krener  shows in \cite{Kre97} that each element of a reciprocal class is characterized through a set of differential characteristics. In this setting, reciprocal characteristics provide the acceleration terms, and are therefore connected to second order expansions, see \cite[Thm 2.1]{Kre97}. This is certainly not possible in the graph case. Indeed to capture the  closed walk characteristic one has to expand up to the length of the walk which, apart from trivial cases, is always at least three. However, some analogies are still present. Indeed, it is a commonly accepted interpretation that, in the case when $j$ does not depend on time, the "speed" of the walk, once it is at site $z$, is $\jj(z)$. The reciprocal characteristic $\cchi_{a}[j](t,\zz)$ associated with the $\zz$ is precisely $\jj(z')-\jj(z)$: it is a difference of velocities and hence it has the resemblance of an acceleration. Clearly, this analogy is far from being anything rigorous and we do not make any claim of physical relevance here.
\end{remark}

\section{Proofs of the main results}\label{sec-proofs}

Let $P$ be any Markov walk such that $P\ll R.$ Therefore $P$  admits an intensity $k(t,z \to z')$  and the related   Girsanov formula (see  \cite{Jac75}) is for each $x\in\supp P_0$,
\begin{align}\label{eq-12}
\frac{dP^x}{dR^x}=\1 _{\{\tau=\infty\}}
	\exp \bigg(\sum_{t:X_{t} \neq X_{t^-}}  \log \frac{k}{j} (t, X_{t^-} \to X_t)
	-\int _{0}^1  (\kk-\jj)(t,X _{t^-})\, dt \bigg)
\end{align}
where  the stopping time $ \tau$ is given by
\begin{alignat}{2}\label{eq-45}
\tau:=\inf \Big\{t\in [0,1] ; k(t,X _{t^-} \to X_t)=0 
\textrm{ or } \int _{0}^t \kk(s,X _{s})\, ds= \infty
\Big\} \in\ii\cup \left\{\infty\right\}
\end{alignat}
with the convention $\inf \emptyset=\infty.$ 
Note that $R^x$-almost surely  $j(t,X _{ t^-}\to X_t)>0,$ for all $ t\in [0,1)$ and $\int _{0}^1 \jj(t,X _{ t^-})\, dt< \infty.$  
\\

We start with a lemma, where we exploit h-transform techniques.

\begin{lemma}[HJB equation]\label{res-03}
For any $x,y \in \XX$ the intensity $j^{xy}$ of the $R^{xy}$-bridge is given by
\begin{equation}\label{eq-43}
j^{xy}(t,\zz) = \exp( \phi_t(z') - \phi_t(z) )j(t,\zz)
\end{equation}
where $\phi$ is the unique classical solution of the HJB-type equation:
\begin{equation}\label{eq-19}
 \partial_t \phi_t(z)+\sum _{ z': \zz}  j(t,\zz)\,[e ^{\phi_t(z')-\phi_t(z)}-1]=0, \quad t,z \in (0,1) \times \XX \\[7pt]
 \end{equation}
subject to the boundary conditions
\begin{equation}
 \lim _{t\to 1^-} \phi_t(z) = \begin{cases} -\log R^{x}_1(y),\quad & \mbox{if $z=y,$} \\ - \infty, \quad & \mbox{otherwise.} \end{cases}
\end{equation}
\end{lemma}
As usual, HJB is a shorthand for Hamilton-Jacobi-Bellman.

\begin{proof}
We first note that 
\bes 
R^{xy} = h(X_1)R^x, \quad h(X_1):= \frac{1}{R^x_{1}(y)} \mathbf{1}_{ \{X_1=y \}}.
\ees
Thanks to our Assumptions \ref{as-03} on the graph and \ref{as-01} on the intensity, the function $h_t(z):=E _{ R^x}[h_1(X_1)\mid X_t=z]$ 
is everywhere well-defined and positive on $(0,1) \times \XX$. Moreover, because of item (2) of Assumption \ref{as-01}, $h$ is continuously differentiable as well.
It is then a well known fact that $h$ is space-time harmonic. This means that it is the unique classical solution of the Kolmogorov  equation
\begin{equation*}
\begin{cases} 
\partial_t h_t(z)+ \sum_{z': z \to z'} j(t,\zz)[h_t(z')-h_t(z)]=0, \quad (t,z) \in (0,1) \times \XX,\\
\lim_{t \rightarrow 1^-} h(t,z) = h(z), \quad z \in \XX.
\end{cases} 
\end{equation*}
Thanks to the positivity and regularity of $h$, we can consider its logarithm $\phi:= \log h$ to obtain, after some standard computation
that $\phi_t$ solves \eqref{eq-19}. Moreover, because of the definition  of $h$ we also have that $\phi_{0}(x) = 0$.
Let us define the intensity
\bes
k(t, \zz ) = \exp(\phi_t(z') - \phi_t(z) ) j(t,\zz).
\ees
Using It\^o formula,
\begin{eqnarray*}
\frac{dR^{xy}}{dR^x}  = \mathbf{1}_{\{X_1=y\}} \frac{1}{R^x_1(y)}  &=& \mathbf{1}_{\{X_1=y\}} \exp \left(  \phi_1(y) - \phi_0(x) \right) \\
& =&\mathbf{1}_{\{X_1=y\}} \exp \left(  \int_{0}^1 \partial_t \phi_t(X_{t^-})dt + \sum_{t: X_{t^-} \neq X_t } \phi_{t}(X_{t}) - \phi_{t}(X_{t^-}) \right)\\
 \end{eqnarray*}
and using Eq. \eqref{eq-19}, we can rewrite
\begin{eqnarray*} 
  \int_{0}^1 \partial_t \phi_t(X_{t^-})dt 	  &= &   -\int_{0}^1 \sum_{z':X_{t^-} \to z'} \exp(\phi_t(z')-\phi_t(X_{t^-}) )( j(X_{t^-} \to z') -1) dt  \\
& = & -\int_{0}^1 \kk(t,X_{t^-}) -\jj(t,X_{t^-}).
\end{eqnarray*}
Moreover, 
\begin{eqnarray*}
&&\sum_{t:X_{t^-} \neq X_t} \phi_{t}(X_{t}) - \phi_{t}(X_{t^-})  \\
&=&\sum_{t:X_{t^-} \neq X_t} \log[j(t,X_{t^-} \to X_{t})\exp(\phi_{t}(X_{t}) - \phi_{t}(X_{t^-}))]  - \log j(t,X_{t^-} \to X_{t})  \\
&=&\sum_{t:X_{t^-} \neq X_t} \log k(t, X_{t^- } \to X_{t} )   - \log j(t,X_{t^-} \to X_{t}) .
\end{eqnarray*}
Therefore,
\begin{equation*}
\frac{dR^{xy}}{dR^x}=\1 _{\{ X_1=y\}}
	\exp \bigg(\sum_{t:X_{t^-} \neq X_t}   \log \frac{k}{j} (t, X_{t^-} \to X_t) -\int _{0}^1  (\kk-\jj)(t,X _{t^-})\, dt \bigg).
\end{equation*}
As $\phi_t$ solves the HJB equation, with  the definition  \eqref{eq-19} of $k$, we see that the event $(\tau=+\infty)$, recall \eqref{eq-45}, coincides $R^x \as$ with the event $(X_1=y).$
An application of Girsanov's theorem allows to identify the intensity $j^{xy}$ of $R^{xy}$ with $k$. 
 \end{proof}

\subsection*{Proof of Theorem \ref{res-08}}

\Boulette{$(\Rightarrow)$}
Let us  show that $P\in\Rec(j)$ shares the announced properties. If $P \in \Rec(j)$, then $P^{xy}=R^{xy}$, which implies that $k^{xy} = j^{xy}$. Item (i) follows from Lemma \ref{res-03}. In particular, since $\phi$ is a classical solution to equation \eqref{eq-19}, the desired regularity in time follows. Item (ii) and (iii) can be proven by replacing $k^{xy}$ with $j^{xy}$. To prove item (ii) we observe that, using \eqref{eq-43}, \eqref{eq-19} can be equivalently be written as
\bes
 \partial_t \phi_t(z)  = \sum_{z':z \to z'} j(t,\zz) - j^{xy}(t,\zz) = \jj(t,z)-\bar{j}^{xy}(t,z),
\ees
and therefore
\begin{eqnarray*}
\partial_t \log j^{xy}(t,\zz) &=& \partial_t \log j(t,\zz) + \partial_t \phi_t(z') - \partial_t \phi_t(z) \\
&=& \partial_t \log j(t,\zz) + \jj(t,z') - \jj^{xy}(t,z') -\jj(t,z)+\jj^{xy}(t,z),
\end{eqnarray*}
from which (ii) follows. Item (iii) is a direct consequence of the gradient appearing in \eqref{eq-43}. Indeed, for any closed walk $\cc =(x_0 \to .. \to x_{n} = x_0)$ and any $t \in (0,1) $ we have
\begin{eqnarray*}
\prod_{i=0}^{n-1} j^{xy}(t, x_{i} \to x_{i+1} ) &=& \prod_{i=0}^{n-1}j(t,\zz) \exp(\sum_{i=0}^{n-1} \phi_t(x_{i+1}) - \phi_t(x_{i})  ) \\
&=& \cchi_c[j](t,\cc) \exp(\phi_t(x_n) - \phi_t(x_0) )=
\cchi_c[j](t,\cc),
\end{eqnarray*}
where the last identity follows from $x_n=x_0.$
\Boulette{$(\Leftarrow)$}
Fix $(x,y) \in \supp P_{01}$ and consider the bridge $k^{xy}$.
Thanks to item (i) we know that $P^{xy} \ll R^{xy}\ll R^{x}$. Therefore the Girsanov formula \eqref{eq-12} applies and in restriction   to $[0,t]$, we have
\begin{align*}
\frac{dP^{xy}_{[0,t]}}{dR^x_{[0,t]}}=\1 _{\{\tau = + \infty \}}
	\exp \bigg(\sum_{0<r<t: X_{r} \neq X_{r_-}}  \log \frac{k^{xy}}{j} (r, X_{r^-} \to X_r) -\int _{0}^{t} (\kk^{xy}-\jj)(r,X _{r^-})\, dr  \bigg)
\end{align*}
where $\tau$ is defined as
\begin{alignat*}{2}
\tau:=\inf \Big\{r \in [0,t]; k^{xy}(r, X_{r^-} \to X_r)=0 
\textrm{ or } \int _{ [0,r]} \kk^{xy}(s, X_{s^-})\, ds= \infty
\Big\} \in [0,t] \cup \left\{\infty\right\}.
\end{alignat*}
The remainder of the proof is divided into four steps.
\begin{enumerate}
\item[\textit{Step 1.\ }]
We claim that 
\bes
 \{\tau = + \infty \}, \quad R^x \as
\ees
Indeed, by assumption, for any $r \in [0,1)$ and any arc $\zz$,  $k^{xy}(r,z \to z')<+\infty$. Since $(X,\to)$ is of bounded degree, we then have that $\kk^{xy}(r,z)$ is also finite. Using the fact that $k^{xy}(\cdot,\zz)$ is continuous  on  $[0,t]$ and that $R^x \as$  we observe finitely many jumps, we deduce that the integral $\int_{0}^t \kk^{xy}(r,X_{r^-})dr $ is finite $R^x \as$.
Moreover, since by assumption, $k^{xy}(r,\zz)>0$ for all $r \in (0,1) $ and $\zz$, the only possibility to have $k^{xy}(r,X_{r^-} \to X_r) =0$ is to have a jump at $r=0$. But this does not happen $R^x \as$. 
\item[\textit{Step 2.\ }]
In this step we show the equality of all closed walk characteristics and all arc characteristics. To prove that $ \cchi_c[k^{xy}] (t,\cc)= \cchi_c [j](t,\cc)$ on a $\mathcal{T}$-basis of the closed walks implies $\cchi_c[k^{xy}] (t,\cc)= \cchi_c [j](t,\cc)$ for \textit{any} closed walk, we first define  the function
\bes 
\ell(z \to z') = \log j(t,z \to z') - \log k^{xy}(t,z \to z').
\ees
 Item (iii) of Theorem \ref{res-08} can then be rewritten as 
\bes
\forall \cc \in \mathcal{C}, \quad \ell(\cc) = 0
\ees
where
\bes  
\ell(\cc) := \sum_{i=0}^{n-1}  \ell( x_i \to x_{i+1} ) = \log ( \cchi[j](t,\cc)-\cchi[k^{xy}](t,\cc).
\ees 
The conclusion follows with an application of Lemma \ref{res-07}, whose proof is detailed at Appendix \ref{sec-B}. Indeed, the implication $(a) \Rightarrow (b)$ shows that $\ell(\cc) =0$ for any closed walk. 
But $\ell(\cc)$ is exactly $\log \cchi[j](t,\cc) -\cchi[k^{xy}](t,\cc)$ (see Definition \ref{def-02}). 

 Let us turn to the arc characteristics. Consider any arc $z \to z'$. Then, since $\mathcal{C}$ is a $\mathcal{T}$-basis of closed walks, either $\mathbf{f}_{z' \to z} \in \mathcal{C}$ or $\mathbf{f}_{z \to z'} \in \mathcal{C}$, recall Definition \ref{defs-01}. We assume w.l.o.g.\ that $\mathbf{f}_{z' \to z} \in \mathcal{C}$, the other case following with minor modifications.
If $(z =z_0 \to ..\to z_n=z')$ is the only simple $\mathcal{T}$-walk from $z$ to $z'$, we have
\begin{eqnarray}\label{eq-46}
 &{}&\cchi_a[k^{xy}] (t,z \to z') \\
 &=& \nonumber \partial_t \log k^{xy}(t,z\to z') + \kk^{xy}(t,z')-\kk^{xy}(t,z)\\
\nonumber &=& \partial_t \log k^{xy}(t,z\to z') + \sum_{i=0}^{n-1}\kk^{xy}(t,x_{i+1})-\kk^{xy}(t,x_{i})\\
\nonumber &=&\partial_t \log \cchi_{c}[k^{xy}](t,z\to z'\to z) - \partial_t \log \cchi_{c}[k^{xy}](t,\mathbf{f}_{z' \to z}) 
+ \sum_{i=0}^{n-1}\cchi_{a}[k^{xy}](t,z_{i}\to z_{i+1}).
\end{eqnarray} 
Since $\mathbf{f}_{z' \to z},(z \to z' \to z) \in \mathcal{C}$ and $z_i \to z_{i+1} \in \mathcal{T}$ for all $0 \leq i \leq n-1$, we can use (ii),(iii) to conclude that all the terms appearing in \eqref{eq-46} coincide with those obtained by replacing $k^{xy}$ with $j$. Repeating backward the computations that led to \eqref{eq-46} we obtain that $\cchi_a[k^{xy}] (t,z \to z')=\cchi_a[j] (t,z \to z')$, which is the desired result.
\item[\textit{Step 3.\ }] In this step we show that the density $\frac{dP^{xy}_{[0,t]}}{dR^x_{[0,t]}}$ is $X_t$- measurable.
The equality of the closed walk characteristics implies that for all $0<r<t$ there exist a potential $\phi_r: \XX \rightarrow \mathbb{R}$ such that
\begin{equation}\label{e-48}
 \log(k^{xy}(r,\zz) ) = \log(j(r,\zz) ) + \phi_r(z')-\phi_r(z), \quad  \forall \zz.
\end{equation}
The proof of this is standard, see Lemma \ref{res-07}. Because of the regularity of $k^{xy}$ and $j$, we can always choose $\phi_r$ to be continuously differentiable in the time variable.
Plugging the expression for $k^{xy}$ we derived at \eqref{e-48} in (ii) we obtain that $\phi_r$ satisfies the HJB equation \eqref{eq-19} on the restricted domain $[0,t] \times \XX$.
But then,
\begin{eqnarray}\label{e-51}
\nonumber && -\int _{0}^{t} (\kk^{xy}-\jj)(r,X _{r^-})\, dr\\\nonumber &=& \int _{0}^{t}\sum_{z':X_{r^-} \to z'} j(r,X _{r^-}\to z')(1- \exp(\phi_r(z') -\phi_r(X_r^-) ) )\, dr\\
&=&\int _{0}^{t} \partial_r \phi_r(X_{r^-})dr .
\end{eqnarray}
Using \eqref{e-48} and \eqref{e-51} and \emph{Step 1}, Girsanov's formula rewrites as
\bes
\frac{dP^{xy}_{[0,t]}}{dR^x_{[0,t]}} = \exp \left( \int^{t}_0 \partial_r \phi_r(X_{r^-}) dr + \sum_{0<r<t, X_{r^-} \neq X_{r} }   \phi_r(X_r)-\phi_r(X_{r^-})  \right).
\ees
Because of the regularity  in time of $\phi$, we can apply the It\'{o} formula to conclude that
\bes
\frac{dP^{xy}_{[0,t]}}{dR^x_{[0,t]}} = \exp \left( \phi_t(X_t) - \phi_0(x) \right).
\ees
\item[\textit{Step 4.\ }]
In this step we show that the density $\frac{dP^{xy}}{dR^x}$ is $ X_1$-measurable. This is a consequence of the fact that $\frac{dP^{xy}_{[0,t]}}{dR^x_{[0,t]}}$ converges in $L^1$ to $\frac{dP^{xy}}{dR^x}$. Because of \emph{Step 3} we have that for any $r<1$, if $t>r$, $\frac{dP^{xy}_{[0,t]}}{dR^x_{[0,t]}}$ is $\sigma(X_{[r,1]})$ measurable, and therefore so is $\frac{dP^{xy}}{dR^x}$. But then it is measurable with respect to $\bigcap_{r <1}   \sigma(X_{[r,1]}) = \sigma(X_1)$. Since $P^{xy}(X_1=y)=1$, then $P^{xy}$ is exactly $R^{xy}$. The proof is complete.
\hfill$\square$
\end{enumerate}
\subsection*{Proof of Theorem \ref{res-06}}\label{sec-pf-res-06}

\Boulette{(a)}  
 Because of \eqref{eq-03}  and $P \in \Rec(j)$, we have: 
 \begin{equation}\label{eq-04}
 \begin{split}
 P(T^t_1\leq t = h/2  &\mid X_t=z,X _{ t+h}=z',T^t_2>t+ h) \\
& =R(T^t_1\in t+hI \mid X_t=z,X _{ t+h}=z',T^t_2>t+ h).
 \end{split}
 \end{equation}
 Therefore it suffices to do the proof with $R$ instead of $P$.

Recall that for a Poisson process with intensity $\lambda(t)$ the density of the law of the first instant  of jump is $t \mapsto \lambda(t) \exp(-\int_0^t \lambda(s) \, ds),$ $t\geq 0$. Therefore for all $\tau \in [0,1]$,
\begin{align*}
R(&T^{t}_1 \leq \ t+h\tau, X_{t+h}=z', T^t_2>t+h \mid X_t=z) \\
	=\ &\int _{0}^{\tau h} \jj(t+r,z)\exp \Big(-\int_0 ^{ r}\jj(t+s,z)\,ds\Big) \frac{j(t+r,\zw )}{\jj(t+r,z)}\exp\Big(-\int _{ r}^{ h}\jj(t+s,z')\,ds\Big)\,dr\\
	=\ &h \int_{0}^{\tau} \exp \Big(-\int_{0}^{ hr}\jj(t+s,z)\,ds\Big) j(t+hr,\zw )
		\exp\Big(-\int _{ hr}^{ h}\jj(t+s,z')\,ds\Big)\,dr.
\end{align*}
Using the following expansions as $h$ tends to zero:
\begin{eqnarray*}
\exp \Big(-\int_{0} ^{ hr}\jj(t+s,z)\,ds\Big)&=& 1-\jj(t,z)h r+o(h),\\
\exp\Big(-\int _{ hr}^{ h}\jj(t+s,z')\,ds\Big)&=&1- \jj(t,z')(1-r)h+o(h),\\
j(t+hr,\zw )&=&j(t,\zw )+ \partial_t j(t,\zw  )hr+o(h),
\end{eqnarray*}
we obtain
\begin{align*}
R(T^t_1\in& t+h \tau, X _{ t+h}=z', T^t_2>t+ h \mid X_t=z) \\
	=\ & hj(t,\zw )\int_{0}^{\tau} \left( 1+ h\left[ \frac{ \partial_t j(t,\zw )}{j(t,\zw )}r-\jj(t,z)r-\jj(t,z')(1-r)\right] \right) \,dr+o(h^2)\\
	=\ &hj(t,\zw ) \Big(\tau+h \big\{\cchi_a[j](t,\zw )\tau^2/2-\jj(t,z')\tau \big\}\Big)\,dr+o(h^2).
\end{align*}
In particular, with $\tau=1$ this implies that 
\begin{align*}
R(X _{ t+h}=z', T^t_2>t+ h &\mid X_t=z) \\
	&=hj(t,\zw )\big(1+ h\big\{\cchi_a[j](t,\zw )/2-\jj(t,z')\big\}\big)+o(h^2).
\end{align*}
When $\tau = 1/2$ we have
\begin{align*}
R(T^t_1 \leq t+h/2 , X _{ t+h}=z' ,T^t_2>t+ h &\mid X_t=z) \\
	&=hj(t,\zw )\big(1/2+ h\big\{\cchi_a[j](t,\zw )/8-\jj(t,z')/2\big\}\big)+o(h^2).
\end{align*}
Taking the ratio of these probabilities leads us to
\begin{align*}
R(T^t_1\leq t+h/2 \mid X_t=z,&X _{t+h}=z',T^t_2>t+ h) \\
= &\frac{1/2+ h\big\{\cchi_a[j](t,\zw )/8-\jj(t,z')/2\big\}+o(h)}{1+ h\big\{\cchi_a[j](t,\zw )/2-\jj(t,z')\big\}+o(h)}\\
	&=  1/2-h \cchi_a[j](t,\zw )/8+o(h).
\end{align*}
With \eqref{eq-04} this gives  \eqref{eq-02}.

\Boulette{(b)}
Using  again \eqref{eq-03} we just need to prove the statement under $R$ rather than $P$.
Since $R(X_t = X_{t+h} =z ) = R(X_t =z)(1 + o(1))$ as $h \to 0^+$, we can write the proof with $R(\cdot \mid X_t =z)$ instead of $R( \cdot \mid X_t = X_{t+h}= z)$. Therefore, if $\cc =(z=x_0 \to x_1\cdots\to x_{|\cc|}=z )$,
\begin{equation*}
\begin{split}
R\Big((&X_t\to X _{ T^t_1}\to\cdots\to X _{ T^t _{ |\cc|}})=\cc,T^t _{ |\cc|}<t+h< T^t _{ |\cc|+1}\mid  X_t=z\Big)\\
	&= \int _{ \{t<t_1<\cdots<t _{ |\cc|}<t+h\}} \prod _{ i=1} ^{ |\cc|}\exp \left[-\int _{ t _{ i-1}}^{ t_i}\jj(s,x _{i-1})\,ds\right] j(t_i,x _{ i} \to x_{i+1}) \\
	&\hskip 8cm \times\  \exp \left[-\int _{ t _{ |\cc|} }^{ t+h}\jj(s,z)\,ds\right] \ dt_1\cdots dt _{ |\cc|} \\
	&= \cchi_c[j](t,\cc)(1 + o(1) ) \
	 \int _{ \{t<t_1<\cdots<t _{ |\cc|}<t+h\}} \exp \Big[- \sum_{i=1}^{\cc} \int _{t_{i}}^{ t_{i+1}}\jj(s,x_{i-1})\,ds \\
	 &\hskip 10cm -\int _{ t _{ |\cc|} }^{ t+h}\jj(s,z)\,ds\Big] \,dt_1\cdots dt _{ |\cc|} \\
	&= \cchi_c[j](t,\cc)h ^{ |\cc|}/|\cc|! +o(h ^{ |\cc|})
\end{split}
\end{equation*}
where we used the convention that $t_0:=t$. This completes the proof of the theorem. 
\hfil$\square$
\section{Examples}\label{sec-examples}
In this series of examples, we illustrate  Theorem \ref{res-08} improved by Proposition \ref{res-14}. We compute the reciprocal characteristic $\cchi[j]$ and sometimes we consider the characteristic equation \eqref{eq-09}.

\subsection*{Birth and death process}

The vertex set is $ \XX=\mathbb{N}$ with the usual graph structure which turns it into an undirected tree. The reference walk $R$ is governed by the time-homogeneous Markov intensity $j(z\to z+1)= \lambda>0,$ $z\ge 0$ and $j(z\to z-1)= \mu>0,$ $z\ge 1.$ Clearly, the set of all edges $ \mathcal{E}= \left\{(z\leftrightarrow  z+1), z\in \mathbb{N}\right\} $ generates $\mathcal{C}$ and  the characteristics of the reference intensity are
$$ \left\{ \begin{array}{l}
\cchi_a [j](z\to z+1)=\cchi_a[j](z+1\to z)=0,\ z\ge 1,\\
\cchi_a [j](0\to 1)=-\cchi_a(1\to 0)= \mu,\\
\cchi_c[j](z \leftrightarrow z+1 )= \lambda \mu, \ z\ge 0.
\end{array}\right.
$$

\subsubsection*{Time-homogeneous Markov walks  in $\Rec[j]$.}
Let us search for such a  random walk $P$.
We denote $\tilde{\lambda}(z)$ the intensity of $(z\to z+1)$ and $\tilde{\mu}(z+1)$ the intensity $(z+1\to z)$ of the Markov walk $P$.
By Theorem \ref{res-08}, $P \in\Rec[j]$ if and only if
\begin{equation*}
\left\{
\begin{array}{l}
\tilde{\lambda}(z+1)+ \tilde{\mu}(z+1)- \tilde{\lambda}(z)- \tilde{\mu}(z)=0,\ z\ge 1,\\
\tilde{\lambda}(1)+ \tilde{\mu}(1)- \tilde{\lambda}(0)= \mu,\\
\tilde{\lambda}(z) \tilde{\mu}(z+1)= \lambda \mu,\ z\ge 0.
\end{array}\right.
\end{equation*}
The solutions to the the above set of equations can be parametrized by choosing $\tilde{\lambda}(0)$ arbitrarily and finding $\tilde{\lambda}(z+1)$, $\tilde{\mu}(z+1)$ recursively as follows 
$$
\left\{ \begin{array}{l}
 \tilde{\mu}(z+1)= \tilde{\lambda}(z)^{-1} \lambda \mu,\ z\ge 0,\\
\tilde{\lambda}(z+1) = \mu + \tilde\lambda(0)-\tilde{\mu}(z+1),\ z\ge 1. \\
\end{array}\right.
$$
With some simple computations one can see that for any   large enough $\tilde{\lambda}(0)$,  the above system admits a unique positive and bounded solution. Hence,  the corresponding Markov walk has its sample paths in $\OO$ and it is in $\Rec[j]$.

\subsection*{Hypercube}
Let $\XX =\{0,1\}^d$ be the $d$-dimensional hypercube with its usual directed graph structure   and let $\{g_i \}_{i=1}^d$ be the canonical basis. For $x \in \XX$, we set $x^i:=x+g_i$ and $x ^{ ik}=x+g_i+g_k$ where we consider the addition modulo $2$. 
Let
$$ \mathcal{S}:= \big\{ (x \to x^i \to x^{ik} \to x^k \to x),  \ x \in \XX, \ 1 \leq i,k \leq d \big\} 
$$
be the set of all directed squares.
One can see that for any given tree $\mathcal{T}$, a $\mathcal{T}$- basis of closed walks can be constructed by selecting only closed walks from the set
\begin{equation}\label{eq-22}
 \mathcal{S} \cup \mathcal{E}.
\end{equation}
This gives a practical canonical way to check the conditions of Theorem \ref{res-08}.

\subsubsection*{The bridge of a simple random walk on the discrete hypercube}

Let $j$ be the simple random walk on
the hypercube. The intensity $j^{xy}(t,\zw )$ of the $xy$-bridge can  be computed  since the transition density of the random walk is known explicitly. 
We have
\begin{equation}\label{eq-36}
j^{xy}(t,z \to z^i) = \begin{cases}
\cosh(1-t)/\sinh(1-t), & \quad \mbox{if $z_i \neq y_i$,} \\
\sinh(1-t)/\cosh(1-t), & \quad \mbox{if $z_i = y_i$,}
 \end{cases} 
\end{equation} 
where $z_i$ and $y_i\in \left\{0,1\right\} $ are the $i$-th coordinates of $z$ and $y\in\XX.$
\\
We provide an alternate proof based on the characteristic equation \eqref{eq-09}.
First, it is immediate to see that under any bridge, all arcs of the hypercube are active at any time. From $\chi_c[j^{xy}] = \chi_c [j]$, we deduce that  the arc function $\log (j/j^{xy})(t,\cdot)$ is the gradient of some potential $\psi_t$, see Lemma \ref{res-07}. The  equality of the arc characteristics  implies that for all $ t\in (0,1) $ and $z\in \mathcal{X}$
\begin{equation*} 
\partial_t \psi(t,z) + \sum_{i=1}^d  [\exp(\psi_t(z^i) - \psi_t(z))-1 ] = \\ \partial_t \psi(t,x) + \sum_{i=1}^d  [\exp(\psi_t(x^i) - \psi_t(x))-1 ] .
\end{equation*}
Since $\psi$ is defined up the addition of a function of time, we can assume without loss of generality that for all $0<t<1,$
$	
\partial_t \psi_t(x) + \sum_{i=1}^d[\exp(\psi_t(x^i) - \psi_t(x))-1 ]=0.
$	
Hence $\psi$ solves the HJB equation
\begin{equation}\label{eq-37}
\partial_t \psi(t,z) + \sum_{i=1}^d  [\exp(\psi_t(z^i) - \psi_t(z))-1 ] =0, \quad  t\in [0,1), z\in \mathcal{X}.
\end{equation}
The boundary data for $\psi$ are
\begin{equation}\label{eq-38}
\lim_{t \rightarrow 1} \psi_t(z) = 
 \begin{cases}  - \infty , & \mbox{if $z \neq y,$} \\ 0 ,  & \mbox{if $z =y.$} \end{cases} 
\end{equation}
One can check with a direct computation that the solution of \eqref{eq-37} \& \eqref{eq-38}  is
\begin{equation}\label{e5}
 \psi(t,z)= \sum_{i=1}^d \log[1 + (-1)^{(z_i - y_i )} e ^{ 2(1-t)} ]
\end{equation}
where the subtraction is considered modulo two. By the definition of $\psi$, we have
$$j^{xy}(t,z\to z ^i)= j(t,z \to z ^i)\exp(\psi_t(z ^i) - \psi_t(z )) = \exp(\psi_t(z ^i) - \psi_t(z ))$$ and \eqref{eq-36} follows with a simple computation.

\subsection*{A triangle}

We consider the triangle $(a,b,c)$ with arcs between any pair of vertices. We define the reference intensity $j$ to be $1$ on every arc.
Our aim is to compute the intensity $j^{a,b}$ of the  \textit{ab} bridge  through the characteristic equation. 
As a tree, we pick $\mathcal{T}= (a\to b\to c)$. The corresponding equations are
\bes
\begin{cases}
\partial_t \log j^{ab}(t, a \to b) + \jj^{ab}(t,b)-\jj^{ab}(t,a) = 0,\\
\partial_t \log j^{ab}(t, b \to c) + \jj^{ab}(t,c)-\jj^{ab}(t,b) = 0.
\end{cases}
\ees
For the closed walk component, we can choose $\mathcal{C}=\{ (a \to b \to c), (a \to b), (b \to c), (c \to a) \}$. Thus we have
\bes
\begin{cases}
j^{ab}(t,a\to b)j^{ab}(t,b \to c)j^{ab}( c \to a) = 1, \\
j^{ab}(t, a \to b) j^{ab}(t, b \to a ) =1,\\
j^{ab}(t, b \to c) j^{ab}(t, c \to b ) =1,\\
j^{ab}(t, c \to a) j^{ab}(t, a \to c ) =1.
\end{cases}
\ees
Concerning the boundary conditions, we have:
\begin{equation*}
 \lim_{t \rightarrow 1} \int_{0}^t \jj^{ab}(s, *)ds = 
\begin{cases} - \log ( R^{a}_1(b)) + 1, \quad & \mbox{if $*=b$ }, \\
-\infty, \quad & \mbox{otherwise}.
 \end{cases}
\end{equation*}
This system of coupled ODEs can be solved explicitly, using the fact that the transition density can be explicitly computed. One can check directly that if we define the functions
\begin{eqnarray*}
\xi_{0}(t)&:=& 1/3 +2/(3 e^{3/2}) \cos \bigg( \frac{\sqrt{3}}{2} \bigg),\\
\xi_{1}(t)&:=& 1/3 +2/(3 e^{3/2}) \sin \bigg( \frac{\sqrt{3}}{2} - \pi /6 \bigg),\\
\xi_{2}(t)&:=& 1/3 -2/(3 e^{3/2}) \sin \bigg( \frac{\sqrt{3}}{2}+ \pi /6 \bigg),
\end{eqnarray*}
we obtain
\begin{eqnarray*}
\begin{cases}
j^{a,b}(t,a\to b) =j^{a,b}(t,c \to b) =(\xi^2_0 + \xi^2_1+ \xi^2_2)/(\xi_1 \xi_2+\xi_1 \xi_0 + \xi_2 \xi_0 )(t), \\
j^{a,b}(t, b \to a ) = j^{a,b}(t, b \to c ) = 1/j^{a,b}(t,a\to b),\\
 j^{a,b}(t,a\to c) =j^{a,b}(t,c \to a) =1.
 \end{cases} 
\end{eqnarray*}

\subsection*{Complete graph}

The directed graph structure of the complete graph on a finite set $\XX = \{1,...,|\XX|\}$ consists of all the couples of distinct vertices, the set of  arcs is $\Aa=\XXX\setminus \left\{(x,x);x\in\XX\right\} .$
Pick an arbitrary vertex $*\in\XX$ and consider the set
 \begin{equation*}
\Delta:= \left\{(*\to z\to z'\to*); z,z',* \textrm{ distinct}\right\} 
 \end{equation*}
 of all directed triangles containing $*$. Then, as indicates Figure \ref{fig-03}, it can be shown that for any tree $\mathcal{T}$ we can construct an associated basis of closed walks by drawing from the set
$\Delta$.

\begin{figure}[h]
\includegraphics[width=10cm,height=3cm]{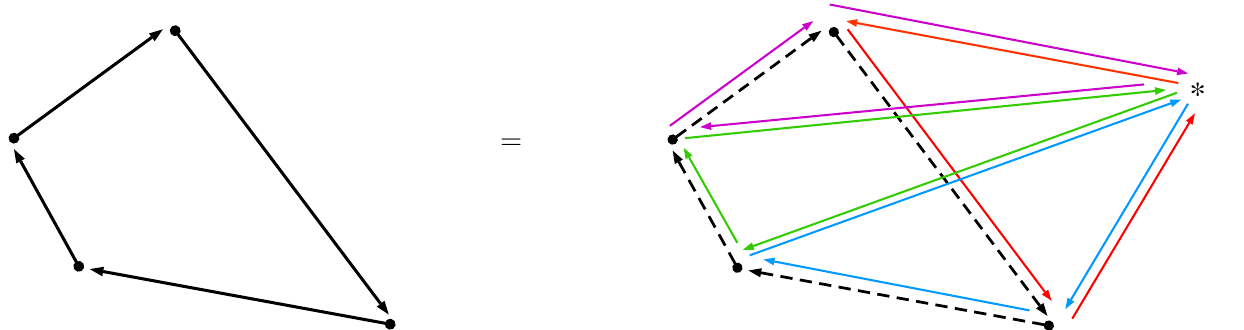}
\caption{Decomposition of a simple walk into  2-walks and 3-walks}\label{fig-03}
\end{figure}

\subsubsection*{Some sampler}
Let us analyze in a bit more detail  one example of a walk on the complete graph.
Take  $ m\in\PX$
a positive probability  distribution  on the finite set $\XX$.  The detailed balance conditions: $ m(z)j(\zz)= m(z')j(z'\to z), \forall z,z',$ tell us that  the intensity
 $$ j(\zw ) = \sqrt{ \frac{ m(z')}{ m(z)}} \quad  z,z' \in \XX $$ admits  $m$ as its reversing measure.
The characteristics associated with $j$  are
\begin{align*}
\chi_a[j](t,(\zw  )) &= 
\Big[\sum _{ x\in\XX}m(x) ^{ 1/2}\Big] (m(z') ^{ -1/2}- m(z) ^{ -1/2})
\\
\chi_{\cc}[j](t,\cc) &= 1
\end{align*}
for any $0\le t\le 1,$ any arc $(\zz)$ and any closed walk $\cc.$

\subsection*{Cayley graphs}

Let $(\XX,*)$ be a group and $ \mathcal{G}= \left\{ g_i; i\in I\right\} $
 be a finite subset of generators. The directed graph structure associated with $ \mathcal{G}$ is defined for any $z,z'\in\XX$ by $\zz$ if  $z'=z g$ for some $g\in \mathcal{G}.$  We introduce the time independent reference intensity $j$ given by 
 \begin{equation*}
 j(z\to zg_i):=j_i,\quad \forall z\in\XX, g_i\in \mathcal{G},
 \end{equation*}
 where $j_i>0$ only depends of the direction $g_i.$ The dynamics of the random walk $R$ is Markov and both time-homogeneous and invariant with respect to the left translations, i.e.\ for all $z_o,z,z'\in\XX,$ $j(z_oz\to z_oz')=j(\zz).$  
For all arc $(\zz)$, we have
\begin{equation*}
\cchi_a(\zz)=0
\end{equation*}
and the closed walk characteristic $\cchi_c$ is  translation invariant. 

\begin{proposition}
Let $j$ and $k$ be two positive  Markov intensities on this Cayley graph which are time-homogeneous and invariant with respect to the left translations. Then, they share the same bridges  if and only if  for any $n\ge 1$ and $(i_1,\dots, i_n)\in I^n$ with $g _{ i_1}\cdots g _{ i_n}=e$, we have $j _{ i_1}\cdots j _{ i_n}=k _{ i_1}\cdots k_{ i_n}.$ 
\end{proposition}

As usual, we have denoted $e$ the neutral element.

\begin{proof}
We have already seen that $\cchi_a[j]=\cchi_a[k]=0.$ On the other hand,
the relation $g _{ i_1}\cdots g _{ i_n}=e$ means that $\cc:=(e\to g _{ i_1}\to g _{ i_1}g _{ i_2}\to \cdots \to g _{ i_1} g _{ i_2}\cdots g _{ i _{ n-1}}\to e)$ is a closed walk and the identity $j _{ i_1}\cdots j _{ i_n}=k _{ i_1}\cdots k_{ i_n}$ means that $\cchi_c[j](\cc)=\cchi_c[k](\cc).$ We conclude with Theorem \ref{res-08}, Proposition \ref{res-14} and the invariance with respect to left translations.
\end{proof}

\begin{remark}
If the group $\XX$ is  Abelian, Proposition \ref{res-15} can be further sharpened by considering only finitely many sequences $(i_1,...,i_n)$. This is done in  \cite[Cor.\,16]{CR14}. The equation (17) in \cite{CR14} corresponds to the equality of the closed walk characteristics. No mention of the arc characteristic is done, since in the case when the jump intensities are translation invariant, they are always zero.
\end{remark}

\subsubsection*{Triangular lattice}

The triangular lattice is the Cayley graph generated by $g_i = ( \cos(\frac{2\pi}{ 3}(i-1), \sin(\frac{2\pi}{ 3}(i-1)),$  $i=1,2,3,$ and we consider a time-homogeneous and translation invariant Markov intensity $j$.   
\begin{figure}[h]
\includegraphics[width=12cm,height=6cm]{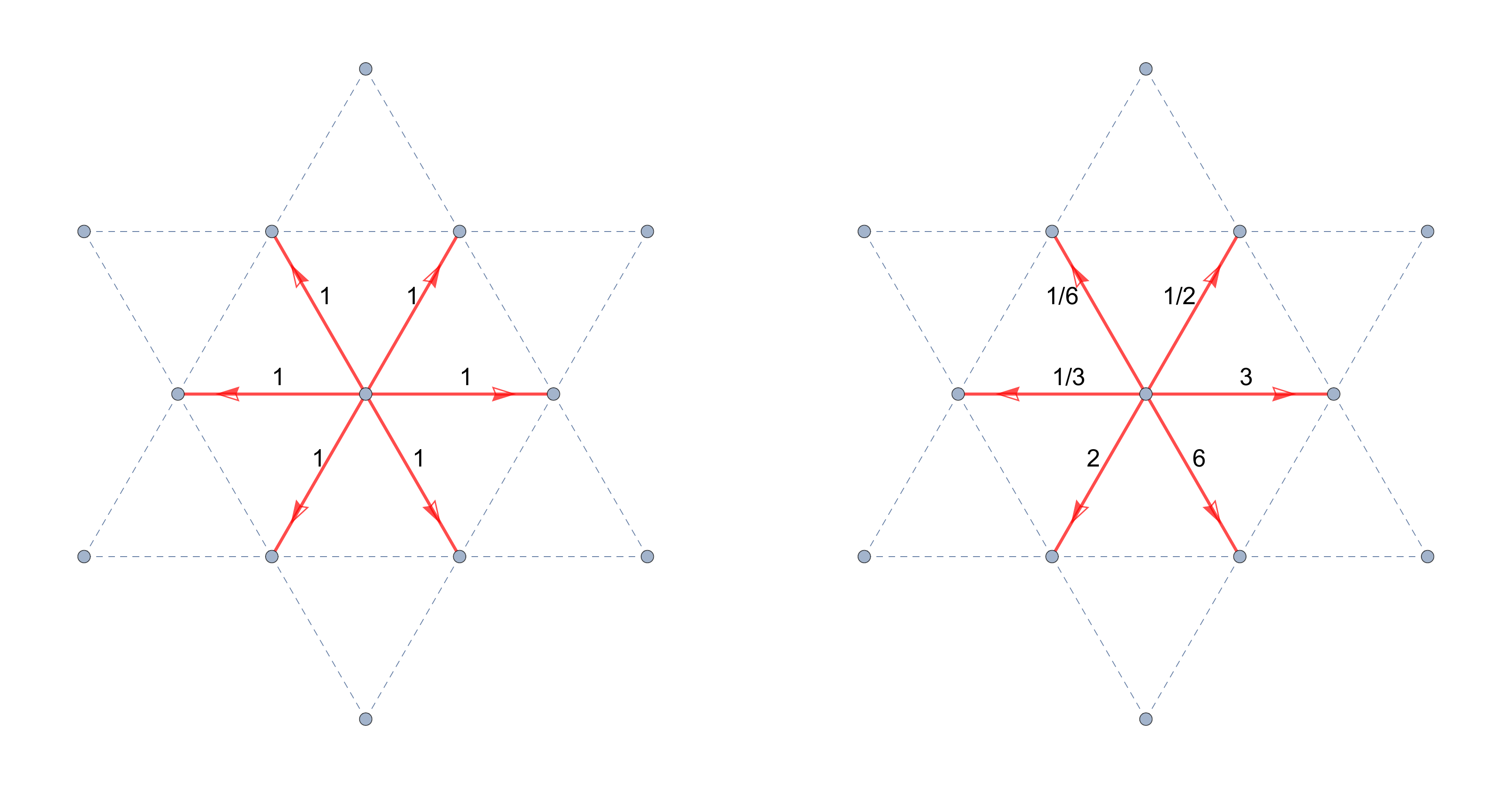}
\caption{Two different space-time homogeneous random walks on the triangular lattice which belong to the same reciprocal class}
\end{figure}

\begin{figure}[h]
\includegraphics[width=12cm,height=6cm]{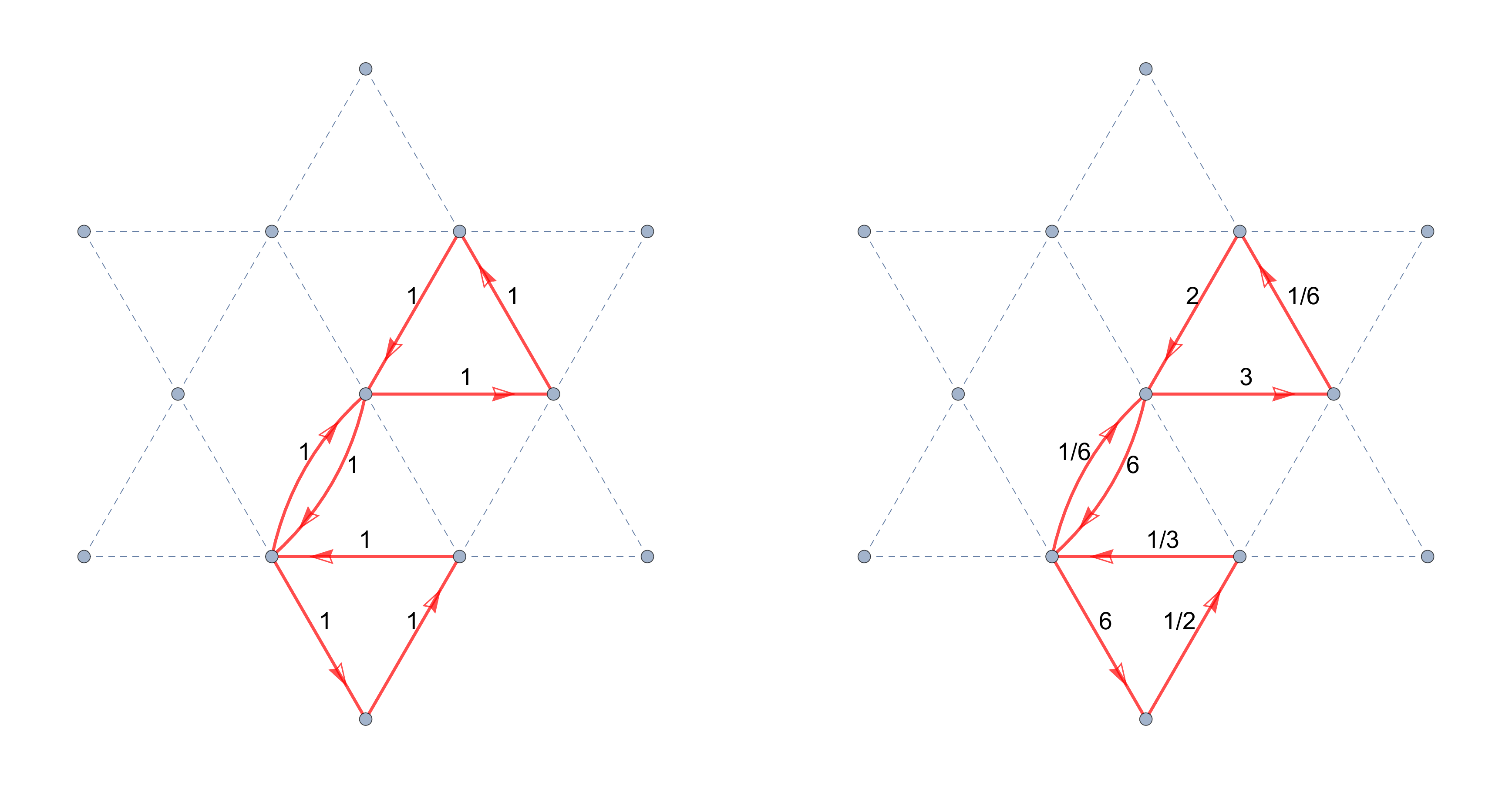}
\caption{The closed walk characteristics coincide}
\end{figure}
For any closed walk $(z\leftrightarrow z+g_i)$  associated with an edge, we have 
  $$
  \chi_{\cc}[j](t,z\leftrightarrow z+g_i) = j_i  j_{-i}.
  $$ 
If we take any counterclockwise oriented face, i.e.\ a closed walk of the form $ \Delta	_z:= (z \to z+ g_1 \to z + g_1+g_2 \to z)$ for $z \in X$ we have  
$$
\chi_{\cc} [j](t, \Delta_z) = j_1j_2j_3.
$$
 We address the  question of finding another space-time homogeneous assignment  $\{ k_{\pm i}\}_{1 \leq i \leq 3}$ such that the corresponding walk  belongs to $\Rec(j)$. Applying Theorem \ref{res-08} or invoking Proposition \ref{res-15}, we can parametrize the solutions $k$ as follows
   \begin{equation*}
   \left\{
   \begin{array}{lcllcl}
   k_1 &=& \alpha j_1,&k_{-1} &=& \alpha^{-1} j_{-1} \\
    k_2 &=& \beta j_2,  &k_{-2}&=& \beta^{-1} j_{-2}\\
k_3  &=&  (\alpha\beta ) ^{ -1} j_3,\quad  &k_{-3}  &=&  \alpha \beta  j_{-3}
   \end{array}
\right.
 \end{equation*}
 where $ \alpha, \beta>0.$
Corollary \ref{res-12} gives some details about the dynamics of the bridge $R^{xy}$ as the unique Markov walk (modulo technical conditions) that starts in $x$, ends in $y$ and such that, if $h>0$ is a very small duration:

\begin{enumerate}
\item At any time $t$ and independently from the current state, it goes back and forth along the direction $i$ during $[t,t+h]$ with probability $ j_i j_{-i} h^2 /2 +o(h^2).  $
\item At any time $t$ and independently from the current state, it goes around the perimeter of a triangular cell of the lattice in the counterclockwise sense during $[t,t+h]$ with probability $j_1 j_2 j_{3} h^3/6 +o(h^3). $ 
\item If exactly one jump occurs during  $[t,t+h]$, then the density of the instant of jump  is constant up to a correction factor of order $o(h)$. This follows from $\chi_a[j](t,\zz)=0$ for all $t$ and $(\zz)$.
\end{enumerate}

\subsubsection*{The lattice $ \mathbb{Z}^d$}

The usual directed graph structure on the vertex set $\XX= \mathbb{Z}^d$  is the Cayley graph structure generated by $ \mathcal{G}= \left\{g_i, g _{ -i}; 1\le i\le d\right\} $ with  $g_i=(0,\dots,0,1,0,\dots,0)$  where $1$ is the $i$-th entry and we denote $g _{ -i}=-g_i$. As another  consequence of Proposition \ref{res-15} we obtain the following 

\begin{corollary}\label{res-16}
Two time-homogeneous and translation invariant positive Markov intensities $j$ and $k$ on $ \mathbb{Z}^d$ generate the same bridges if only if for all $1\le i\le d,$ they satisfy 
\begin{equation*}
j_i j _{ -i}=k_ik _{ -i},\quad \forall 1\le i\le d
\end{equation*} 
where $j _{ -i}$ and $ k _{ -i}$ are the intensities of jump in the direction $g _{ -i}=-g_i.$
\end{corollary}

\begin{proof}
This set of equalities corresponds to the identification of the closed walk characteristic along the edges. But, because the group is Abelian, it also implies the identification along the squares, which is enough to conclude with \eqref{eq-22}.
\end{proof}

\subsubsection*{Hypercube, again}

Let us visit once more the hypercube $\XX= (\mathbb{Z}/2 \mathbb{Z})^d$ which is seen now as the Cayley graph generated by the canonical basis $g_i=(0,\dots,0,1,0,\dots,0),$ $1\le i\le d,$  where $1$ is the $i$-th entry. As another  consequence of Proposition \ref{res-15} we obtain the following 

\begin{corollary}
Two time-homogeneous and translation invariant positive Markov intensities  on the hypercube generate the same bridges if only if  they coincide.
\end{corollary}

\begin{proof}
The proof is the same as Corollary \ref{res-16}'s one. But this time $g _{ -i}=g_i,$ so that $j_i j _{ -i}=k_ik _{ -i}$ is equivalent to $j_i^2=k_i^2.$
\end{proof}

\appendix

\section{Shared bridges}\label{sec-A}

Recall that the reciprocal class is given by
 \begin{equation}\label{eq-x05}
 \Rec(j):= \left\{\sum _{ x,y\in\XX} \pi(x,y)\,R ^{ xy}; \pi\in\PXX: \supp\pi\subset\supp R _{ 01}\right\} \subset \PO.
\end{equation}
 Clearly, if $P \in \Rec(j)$ shares its bridges with $R$. Next proposition gives  assertions which are equivalent to this property. It is crucial for the proof of Theorem \ref{res-08}.
 
  \begin{proposition}\label{res-02}
For any random walks $P,R\in\PO$ such that $\supp P _{ 01}\subset\supp R_{ 01},$ the following assertions are equivalent.
\begin{enumerate}[(a)]
\item
$P = \Rec(j)$.
\item
$P ^{ xy}=R ^{ xy}$ for all $(x,y)\in\supp P _{ 01}.$
\item
There exists a measurable function $k:\XXX\to [0,\infty)$ such that $\IXX k\, dR _{01}=1$ and 
\begin{equation*}
P=k(X_0,X_1)\,R.
\end{equation*}
\item
There exists a measurable function $h:\XXX\to [0,\infty)$ such that  for all $x\in \supp P_0,$ 
\begin{equation*}
P^x=h(x,X_1)\,R^x.
\end{equation*}
\item For all $x \in \supp P_0$, $P^x \in \Rec(j)$.
\end{enumerate}
Moreover, the function $h$ at (d) is given by $h(x,y)=\displaystyle{ \frac{dP_1^x}{dR_1^x}(y)}$.

\end{proposition}

\begin{proof}\ 
$[(a)\Leftrightarrow (b)]:$\ Take $ \pi=P _{ 01}.$\\
$[(b)\Rightarrow (c)]:$\ Take $k=dP _{01}/dR _{01}.$\\
$[(c)\Rightarrow (d)]:$\ Take $h(x,y)=k(x,y)/E _{R^x}k(x,X_1).$
\\
$[(d)\Rightarrow (b)]:$\ As $\supp P _{01}\subset\supp R _{01},$ we have $P ^{ xy}= {\displaystyle \frac{h(x,X_1)}{E _{ R ^{ xy}}h(x,X_1)}\,R ^{ xy}}=R ^{ xy}$ for all $(x,y)\in \supp P _{ 01}.$
\\
$[(d)\Rightarrow (e)]:$\ If $(d)$ holds, then $P^{x} = \sum_{\tilde{x},y \in \XXX} \frac{h(x,y)}{R_0(x)} \mathbf{1}_{x = \tilde{x}} R^{\tilde{x},y} $. This decomposition shows that $P^x \in \Rec(j)$.
\\
$[(d)\Rightarrow (e)]:$ Fix $x \supp P_0$. If $P^x \in \Rec[j]$ then there exist some $\pi$ such that $P$ writes as in \ref{eq-x05}. We can w.l.o.g. assume that $\pi(\tilde{x},y) =0$ whenever $\tilde{x} \neq x$. But then
\bes
P^x = \sum_{y \in \XX} \pi(x,y) R^{xy} = \sum_{y \in \XX} \frac{\pi(x,y)}{R^x_1(y)} \mathbf{1}_{X_1=y} \ R^{x}
\ees
which shows (d), with $h(x,y) = \frac{\pi(x,y)}{R^x_1(y)}$.
\end{proof}

The identity in (d) expresses that $P^x$ is an $h$-transform of $R^x$ in the sense of Doob \cite{Doob57}.

The proofs of Theorem \ref{res-06} and Corollary \ref{res-12} rely on a classical result in the theory of reciprocal processes  which we recall below at Proposition \ref{res-05}. It is a  consequence of the reciprocal property, see Definition \ref{def-10}.

\begin{proposition}\label{res-05}
For any random walk $P\in\PO$, the following statements are equivalent.
\begin{enumerate}[(a)]
\item
$P \in \Rec(j)$.
\item
 For all $0\le s\le t\le1,$  we  have
\begin{equation}\label{eq-03}
P(X _{ [s,t]} \in\cdot\mid X_s,X_t)=R(X _{ [s,t]}\in\cdot\mid X_s,X_t),\quad P \as 
\end{equation}
\end{enumerate}
\end{proposition}

\begin{proof}
Statement (b) with $s=0$ and $t=1$ is nothing but the statement (b) of Proposition \ref{res-02}. Therefore,  (b) $\Rightarrow$ (a).
\\
Let us prove the converse statement: (a) $\Rightarrow$ (b). 
Let $u$ be any bounded function  on $\XXX$ and $B$ be any $X _{ [s,t]}$-measurable event. We have
\begin{eqnarray*}
E_P[P(B\mid X_s,X_t)u(X_s,X_t)]&=&
E_P[\1_B u(X_s,X_t)]\\
&=& \sum _{ x,y\in\XX}\pi(x,y)E _{ R ^{ xy}}[\1_B u(X_s,X_t)]\\
&=& \sum _{ x,y\in\XX}\pi(x,y)E _{ R ^{ xy}}[R ^{ xy}(B\mid X_s,X_t) u(X_s,X_t)]\\
&=& \sum _{ x,y\in\XX}\pi(x,y)E _{ R ^{ xy}}[R (B\mid X_s,X_t) u(X_s,X_t)]\\
&=&E_P[R(B\mid X_s,X_t)u(X_s,X_t)]
\end{eqnarray*}
which implies the announced result. We have used the reciprocal property of $R$ at the last but one equality.
\end{proof}

\begin{remarks}\label{rem-01}
\ \begin{enumerate}[(a)]
\item
Any Markov walk is reciprocal, but the converse fails.  
\item
 Any  bridge of a reciprocal  walk is Markov. 
 \item
For any   reciprocal walk $R$ and any $ \pi\in\PXX$ such that $\supp\pi\subset\supp R _{ 01},$ the mixture of bridges \eqref{eq-x05}
is also a reciprocal walk. 
\item
In particular, the reciprocal class of a Markov measure (for instance $\Rec(j)$) consists of reciprocal measures.
\end{enumerate}
\end{remarks}
For detail about reciprocal path measures, see \cite{LRZ12} and the references therein.

\section{Closed walks}\label{sec-B}

Closed walks are necessary to define the closed walk component $\cchi_c$ of the reciprocal characteristic, see Definition  \ref{def-01}-b.

\begin{definitions}[Walk, closed walk, simple closed walk and gradient] \label{def-02}
Let $ \AA\subset\XXX$ specify a directed graph $(\XX,\to)$ on $\XX$ satisfying Assumption \ref{as-03}.
\begin{enumerate}[(a)]

\item
For any $n\ge1$ and $x_0,\dots,x_n\in\XX$ such that $x_0\to x_1$, $x_1\to x_2,\ \cdots,\ x _{ n-1}\to x_n$, the ordered sequence
 $(x_0,x_1,\dots,x_n)$ is called a \emph{walk}.
We adopt the more appealing notation $	
\ww=(x_0\to x_1 \to \cdots\rightarrow x_n).
$	
The  length $n$ of $\ww$  is denoted by $|\ww|.$
\item If $w=(x_{0} \to x_1 \to \cdots \to x_n)$ is a walk, then $w^*$ is the walk obtained by reverting the orientation of all arcs:
\begin{equation}\label{eq-60}
(x_n \to x_{n-1} \to \cdots \to x_{0}).
\end{equation}

\item
When $x_n=x_0$,  the  walk $\cc=(x_0\to x_1 \to \cdots\rightarrow x_n=x_0)$ is said to be \textit{closed}. 
\item
A closed walk $\cc=(x_0\to x_1 \to \cdots\rightarrow x_n=x_0)$ is said to be simple if the cardinal  of the visited vertices $ \left\{x_0,x_1,\dots, x _{ n-1}\right\} $ is equal to the length $n$ of the walk. This means that a simple closed walk cannot be decomposed into several closed walks.
\item
An arc function $\ell: \AA\to\RR$ is  the \emph{gradient of the  vertex function} $\psi:\XX\to\RR$ if
\begin{equation}\label{eq-24}
\ell(\zz)=\psi(z')-\psi(z),\quad \forall (\zz)\in \AA.
\end{equation}
For any arc function $\ell$, we denote  
$	
\ell(\ww):=\ell(x_0\rightarrow x_1)+\cdots+\ell(x _{n-1}\rightarrow x_n).
$
\end{enumerate}
\end{definitions}

\begin{lemma}\label{res-07}

Let $(\XX,\to)$ satisfy Assumption \ref{as-03} and let $\mathcal{T}$ be a tree and $\mathcal{C}$ be a $\mathcal{T}$-basis of closed walks as in Definition \ref{defs-01}.
The following assertions are equivalent.
\begin{enumerate}[(a)]
\item The arc function $\ell: \AA\to\RR$ is a gradient in the sense of \eqref{eq-24}.
\item For \textit{any} closed walk $\cc$, $\ell(\cc)=0$.
\item For all closed walks in $\mathcal{C}$, $ \ell(\cc )=0 $.

\end{enumerate}
\end{lemma}

\begin{proof} 
$(a )\Rightarrow (b ).$ \,This is standard. If $\ell$ is the gradient of $\psi,$ then $\ell(\ww)=\psi(x_n)-\psi(x_0)$, which vanishes when $\ww=(x_0\to\cdots\to x_n)$ is a closed walk.\\
$(b) \Rightarrow (a).$\,
 Let $\ell$ be such that $\ell(\cc)=0$ for all closed walks $\cc.$ As $(z\rightarrow z'\rightarrow z)$ is a closed walk, we have 
\begin{equation}\label{eq-23}
\ell(z\rightarrow z')+\ell(z'\rightarrow z)=0,\quad \forall z\leftrightarrow z'\in\XX.
\end{equation}
Choose a tagged vertex $\ast\in \XX,$
set $\psi(\ast)=0$ and for any $x\not=\ast,$ define
$$
\psi(x):=\ell(\ww),\quad \textrm{ for any }\ww\in \left\{(\ast\rightarrow x_1\rightarrow \cdots\rightarrow x_n=x),  \textrm{ for some } n\ge1\right\}.
$$ 
To see that this is a meaningful definition, take two paths $\ww=(\ast\rightarrow x_1\cdots\rightarrow x_n)$ and $\ww'=(\ast\rightarrow y_1\cdots\rightarrow y_m)$ such that $x_n=y_m=x$. As $(\ast\rightarrow x_1\cdots\rightarrow x_n=x=y_m\to y _{m-1}\rightarrow\cdots\rightarrow \ast)$ is a closed walk, we have $0=\ell(\ast\rightarrow x_1\cdots\rightarrow x)+\ell(x\rightarrow y _{m-1}\to\cdots\rightarrow \ast)
	=\ell(\ww)-\ell(\ww'),$ where the last equality is obtained with \eqref{eq-23}. Therefore, $ \psi$ is well defined.
Finally, it follows immediately from our definition of $\psi$ that $\ell(\zz) = \psi(z')- \psi(z),$ for all $(\zz)$.
\\
$(b) \Leftarrow (c).$ \,
Consider an arbitrary closed walk $(x_0 \to x_1 \to .. \to x_{n-1} \to x_n=x_0)$. The proof is by induction on the number of arcs which are not in $\mathcal{T}$. We can w.l.o.g. assume that $x_{n-1} \to x_n $ is one of such arcs. Because of the fact that $\mathcal{T}$ is a spanning tree, there exist a directed walk $\ww$ from $x_{n-1}$ to $x_0$ which uses only arcs in $\mathcal{T}$. Then the closed walk $\ff$ obtained by concatenating $x_0 \to ..\to x_{n-1}$ and $\ww$ s a closed walk which uses strictly less arcs not in $\mathcal{T}$ than $\cc$. Moreover, the walk obtained concatenating $x_{n-1} \to x_0$ with $\ww^*$ belongs to $\mathcal{C}$, by definition of $\mathcal{C}$. Indeed, such a walk is precisely $\ff_{x_{n-1} \to x_0}$, see Definition \ref{defs-01}. By definition of $\mathcal{C}$ we have that either $\ff_{x_{n-1} \to x_0 } \in \mathcal{C}$ or $\ff_{x_{0} \to x_{n-1} }\in \mathcal{C}$ . W.l.o.g. we can assume that the first holds (the other case can be analyzed similarly by using the fact that $\mathcal{E} \subseteq \mathcal{C}$). But then,
\begin{alignat*}{1}
\ell(\cc) &= \ell(\ff) \underbrace{- \ell(\ww)}_{= \ell(\ww^*)} + \ell(x_{n-1} \to x_{n}) = \ell(\ff) + \ell(\ww^*) + \ell(x_{n-1} \to x_{n})\\
& = \underbrace{\ell(\ff)}_{=0 \, \textrm{ by induction}} + \underbrace{\ell(\ff_{ x_{n-1} \to x_0})}_{\mathbf{f}_{x_{n-1} \to x_0}\in \mathcal{C}} =0
\end{alignat*}
where to conclude that $\ell(\ww^* ) = -\ell(\ww)$ we used the fact that $\mathcal{E} \subset \mathcal{C}$.
The proof is now complete.\\
$(c) \Rightarrow (a)$ is obvious.
\end{proof}


\begin{thebibliography}{CLMR15}

\bibitem[Ber32]{Bern32}
S.~Bernstein.
\newblock Sur les liaisons entre les grandeurs al\'eatoires.
\newblock {\em Verhand. Internat. Math. Kongr. Z{\"u}rich}, Band I, 1932.

\bibitem[BM07]{BM07}
A.~Bondy and U.S.R. Murty.
\newblock {\em Graph Theory}.
\newblock Graduate Texts in Mathematics. Springer London, 2007.

\bibitem[Cla91]{Cl91}
J.~Clark.
\newblock A local characterization of reciprocal diffusions.
\newblock {\em Applied Stoch. Analysis}, 5:45--59, 1991.

\bibitem[CLMR15]{CLMR14}
G.~Conforti, C.~L{\'e}onard, R.~Murr, and S.~R{\oe}lly.
\newblock Bridges of {M}arkov counting processes. {R}eciprocal classes and
  duality formulas.
\newblock {\em Electr. Comm. Probab.}, 20:Article 18, 12 pp., 2015.

\bibitem[Con]{C15}
G.~Conforti.
\newblock Reciprocal characteristics and concentration of measure.
\newblock In preparation.

\bibitem[CPR]{CDPR14}
G.~Conforti, P.~Dai Pra, and S.~R{\oe}lly.
\newblock Reciprocal class of jump processes.
\newblock {\em To appear in Journal of theoretical Probability}.

\bibitem[CR]{CR14}
G.~Conforti and S.~R{\oe}lly.
\newblock Reciprocal class of random walks on an {A}belian group.
\newblock {\em To appear in Bernoulli}.

\bibitem[CZ08]{CZ08}
K.L. Chung and J.-C. Zambrini.
\newblock {\em Introduction to Random Time and Quantum Randomness}.
\newblock World Scientific, 2008.

\bibitem[Doo57]{Doob57}
J.L. Doob.
\newblock Conditional {B}rownian motion and the boundary limits of harmonic
  functions.
\newblock {\em Bull. Soc. Math. France}, 85:431--458, 1957.

\bibitem[F{\"o}l88]{Foe85}
H.~F{\"o}llmer.
\newblock {\em Random fields and diffusion processes, in \'Ecole d'\'et{\'e} de
  Probabilit{\'e}s de Saint-Flour XV-XVII-1985-87}, volume 1362 of {\em Lecture
  Notes in Mathematics}.
\newblock Springer, Berlin, 1988.

\bibitem[Jac75]{Jac75}
J.~Jacod.
\newblock Multivariate point processes: predictable representation,
  {R}adon-{N}ikod\'ym derivatives, representation of martingales.
\newblock {\em Z. Wahrsch. verw. Geb.}, 31:235--253, 1975.

\bibitem[Jam74]{Jam74}
B.~Jamison.
\newblock Reciprocal processes.
\newblock {\em Z. Wahrsch. verw. Geb.}, 30:65--86, 1974.

\bibitem[Jam75]{Jam75}
B.~Jamison.
\newblock The {M}arkov processes of {S}chr{\"o}dinger.
\newblock {\em Z. Wahrsch. verw. Geb.}, 32(4):323--331, 1975.

\bibitem[KL93]{KL93}
A.~Krener and B.~Levy.
\newblock Dynamics and kinematics of reciprocal diffusions.
\newblock {\em J. Math. Phys.}, 34(5):1846--1875, 1993.

\bibitem[Kre88]{Kre88}
A.~Krener.
\newblock Reciprocal diffusions and stochastic differential equations of second
  order.
\newblock {\em Stochastics}, 24:393--422, 1988.

\bibitem[Kre97]{Kre97}
A.~Krener.
\newblock Reciprocal diffusions in flat space.
\newblock {\em Probab. Theory Relat. Fields}, 107:243--281, 1997.

\bibitem[L{\'e}o14]{Leo12e}
C.~L{\'e}onard.
\newblock A survey of the {S}chr{\"o}dinger problem and some of its connections
  with optimal transport.
\newblock {\em Discrete Contin. Dyn. Syst. A}, 34(4):1533--1574, 2014.

\bibitem[LRZ14]{LRZ12}
C.~L\'{e}onard, S.~R{\oe}lly, and J-C. Zambrini.
\newblock Reciprocal processes. {A} measure-theoretical point of view.
\newblock {\em Probab. Surv.}, 11:237--269, 2014.

\bibitem[LZ]{LZ14}
C.~L{\'e}onard and J.-C. Zambrini.
\newblock Entropy minimization and calculus of variations of diffusion
  processes.
\newblock In preparation.

\bibitem[Nel85]{Nel85}
E.~Nelson.
\newblock {\em Quantum fluctuations}.
\newblock Princeton Series in Physics. Princeton University Press, Princeton,
  NJ, 1985.

\bibitem[RT04]{RT02}
S.~R{\oe}lly and M.~Thieullen.
\newblock A characterization of reciprocal processes via an integration by
  parts formula on the path space.
\newblock {\em Probab. Theory Related Fields}, 123:97--120, 2004.

\bibitem[RT05]{RT05}
S.~R{\oe}lly and M.~Thieullen.
\newblock Duality formula for the bridges of a {B}rownian diffusion:
  {A}pplication to gradient drifts.
\newblock {\em Stochastic Processes and their Applications}, 115:1677--1700,
  2005.

\bibitem[Sch31]{Sch31}
E.~Schr\"odinger.
\newblock {\"U}ber die {U}mkehrung der {N}aturgesetze.
\newblock {\em Sitzungsberichte Preuss. Akad. Wiss. Berlin. Phys. Math.},
  144:144--153, 1931.

\bibitem[Sch32]{Sch32}
E.~Schr{\"o}dinger.
\newblock Sur la th{\'e}orie relativiste de l'{\'e}lectron et
  l'interpr{\'e}tation de la m{\'e}canique quantique.
\newblock {\em Ann. Inst. H. Poincar\'e}, 2:269--310, 1932.

\bibitem[Thi93]{Th93}
M.~Thieullen.
\newblock Second order stochastic differential equations and non-gaussian
  reciprocal diffusions.
\newblock {\em Probab. Theory Related Fields}, 97:231--257, 1993.

\end{thebibliography}
\end{document}